\documentclass[11pt]{amsart}

\usepackage{amssymb}
\usepackage{amsfonts}
\usepackage{amsmath}
\usepackage{url}
\usepackage{color}
\usepackage{setspace}
\usepackage{enumerate}
\usepackage[
pdfborderstyle={},
pdfborder={0 0 0},
pagebackref]{hyperref}

\renewcommand*{\backref}[1]{}
\renewcommand*{\backrefalt}[4]{
  \ifcase #1 %
   [No citations.]%
  \or
   [#2]%
  \else
   [#2]%
  \fi
}

\let\oldmarginpar\marginpar
\renewcommand\marginpar[1]{\oldmarginpar[\raggedleft\footnotesize #1]%
{\raggedright\footnotesize #1}}

\usepackage{epsfig}
\usepackage{graphicx}
\usepackage{psfrag}

\newcommand{\RR}{{\mathbb{R}}}
\newcommand{\ZZ}{{\mathbb{Z}}}

\newcommand{\bdy}{{\partial}}
\newcommand{\vol}{{\mathcal{V}}}
\newcommand{\ang}{{\mathcal{A}}}
\newcommand{\lad}{{\mathcal{L}}}

\newcommand{\eps}{{\varepsilon}}
\newcommand{\ignore}[1]{}

\newtheorem{theorem}{Theorem}[section]
\newtheorem{proposition}[theorem]{Proposition}

\newtheorem{lemma}[theorem]{Lemma}

\newtheorem{claim}[theorem]{Claim}

\theoremstyle{definition}
\newtheorem{define}[theorem]{Definition}
\newtheorem{remark}[theorem]{Remark}
\newtheorem{example}[theorem]{Example}
\newtheorem{observation}[theorem]{Observation}

\numberwithin{equation}{section}

\title{Explicit angle structures for veering triangulations}
\subjclass[2000]{57M50, 57R05}

\author{David Futer}
\address{Department of Mathematics, Temple University,
Philadelphia, PA 19122, USA}
\email{dfuter@temple.edu}
\thanks{Futer is supported in part by NSF Grant No. DMS--1007221.}

\author{Fran\c{c}ois Gu\'eritaud}
\address{Laboratoire Paul Painlev\'e, CNRS UMR 8524, Universit\'e de Lille 1, 59650 Ville\-neuve d'Ascq, France}
\email{Francois.Gueritaud@math.univ-lille1.fr}
\thanks{Gu\'eritaud is supported in part by the ANR program ETTT (ANR-09-BLAN-0116-01).}

\date{\today}

\begin{document}

\maketitle

\begin{abstract}
Agol recently introduced the notion of a \emph{veering} triangulation, and showed that such triangulations naturally arise as layered triangulations of fibered hyperbolic $3$--manifolds. We prove, by a constructive argument, that every veering triangulation admits positive angle structures, recovering a result of Hodgson, Rubinstein, Segerman, and Tillmann. Our construction leads to explicit lower bounds on the smallest angle in this positive angle structure, and to information about angled holonomy of the boundary tori.
 \end{abstract}

\section{Introduction}

Let $\overline{M}$ be a compact oriented 3-manifold whose boundary $ \partial \overline{M}$ is a nonempty union of tori. Suppose $\overline{M}$ is decomposed into finitely many truncated tetrahedra $\Delta_i$, glued in pairs along all their hexagonal faces. The \emph{boundary triangles} or \emph{cusp triangles} obtained by truncating tips of the tetrahedra remain unglued, and give a (possibly non-simplicial) triangulation of $\partial M$. We then say that $\tau = \{ \Delta_1, \ldots \Delta_n \} $ determines an \emph{ideal triangulation} of the noncompact manifold $M=\overline{M}\smallsetminus \partial \overline{M}$ (by removing the cusp triangles). Note that $ \partial \overline{M}$ has two types of edges: cusp edges (contained in $\partial \overline{M}$) and interior edges (contained in $M$).


\begin{define} \label{def:anglestructure}
We will assign a real variable $\theta_j$, called an \emph{angle}, to every pair of opposite edges of every tetrahedron. Thus there are $3n$  variables for the $n$ tetrahedra.  
The angles associated to $\Delta_i$ are $ \theta_{3i-2}, \theta_{3i-1}, \theta_{3i}$.
We impose the following system of equations:
\begin{enumerate}
\item\label{eq:tetrahedron} For each tetrahedron $\Delta_i$, the angle sum is $\theta_{3i-2} + \theta_{3i-1}+ \theta_{3i}=\pi$.
\item\label{eq:edge} For each (interior) edge of $M$ surrounded by angles $\theta_{j_1}, \dots, \theta_{j_s}$, one has $\sum \theta_{j_i} = 2\pi$.
\end{enumerate}
An angle vector $(\theta) = (\theta_1, \ldots, \theta_{3n})$ satisfying equations \eqref{eq:tetrahedron} and \eqref{eq:edge} is called: 
\begin{itemize}
\item a \emph{generalized angle structure} on $\tau$ if  $\theta_j \in \RR$ for all $j$,
\item a \emph{taut angle structure} on $\tau$ if $\theta_j \in \{0,\pi\}$ for all $j$,
\item a \emph{positive angle structure} or simply an \emph{angle structure} on $\tau$  if $\theta_j > 0$ for all $j$. Note that by \eqref{eq:tetrahedron}, this implies $\theta_j \in (0,\pi)$ for all $j$.
\end{itemize}
\end{define}

These related definitions are  ordered according to (roughly) increasing strength. 
Casson, Luo, and Tillmann showed that every ideal triangulation of a manifold with torus boundary admits a generalized angle structure \cite[Theorem 1]{luo-tillmann}. If $M$ is irreducible and acylindrical, Lackenby showed that \emph{some} ideal triangulation of $M$ admits a taut angle structure \cite{lackenby:taut}. On the other hand, positive angle structures are rarer and more powerful: Casson and Lackenby showed that if an ideal triangulation of $M$ admits angle structures, then $M$ must carry a complete hyperbolic metric.
See \cite[Corollary 4.6]{lackenby:surgery} and \cite[Theorem 1.2]{fg:arborescent}.

The theme of this paper is to deform a taut angle structure into a (positive) angle structure, with explicit information about the angles.

\medskip

A taut angle structure endows every boundary  triangle in the tiling of $\partial \overline{M}$ with a triple of angles $(0,0,\pi)$. Therefore, every vertex $v$ in the tiling of $\partial \overline{M}$ has a link consisting of two $\pi$-angles, also called \emph{wide} angles, separated by two (possibly empty) sequences of $0$-angles, also called \emph{thin} angles. These two sequences of thin angles are called the two \emph{fans} adjacent to $v$. \emph{Wide} and \emph{thin} will always refer to the same angles at $v$, even when we will start assigning them other values than $0, \pi$.

A nice source of taut angle structures comes from \emph{layered} triangulations, which are constructed in the following way. Let $S$ be a surface with punctures and $\rho$ an ideal triangulation of $S$. Let $\phi:S\rightarrow S$ be a pseudo--Anosov diffeomorphism, and suppose $M$ is the mapping torus $M=(S\times [0,1])/\!\sim_\phi$,  where $(x,1)\sim_\phi(\phi(x), 0)$. Find a path from the triangulation $\rho$ to its pushforward $\phi_*(\rho)$, via a sequence of \emph{diagonal exchanges}. Each such diagonal exchange can be seen as a flattened tetrahedron, with angles of $\pi$ on the exchanged diagonals and $0$ on the periphery. The union of these tetrahedra gives a taut ideal triangulation of $M$. 

The taut angle structure on a layered triangulation admits an additional global property: 
a coherent choice of transverse orientation  $\sigma$ on the $2$--skeleton, such that for each tetrahedron, the two faces sharing one $\pi$-angle have $\sigma$ pointing inward, and the two faces sharing the other $\pi$-angle have $\sigma$ pointing outward. We call this property \emph{transverse--taut}.\footnote{Following Lackenby \cite{lackenby:taut}, most authors call \emph{taut} what we here call \emph{transverse--taut}, as this notion is inspired by that of a taut foliation. 
By contrast, a taut angle structure can be called \emph{angle--taut}  for short.}
It is not hard to see that if a taut angle structure is not transverse--taut, it must have a transverse--taut double cover; see Lemma \ref{lemma:transverse-cover}. To summarize the implications, 
 $$\text{layered}\Longrightarrow\text{transverse--taut}\Longrightarrow\text{taut angle structure} .$$

Agol recently introduced the following notion  \cite{agol}.
\begin{define} \label{def:veering}
A taut angle structure is \emph{veering} if for every vertex $v$ of $\partial \overline{M}$, either all the triangles of the fans of $v$ have their $\pi$--angle immediately \emph{before} $v$ in the counterclockwise cyclic order, or all the triangles of the two fans of $v$ have their $\pi$-angle immediately \emph{after} $v$ in the counterclockwise cyclic order. We say $v$ is \emph{left--veering} in the first case, \emph{right--veering} in the second. A veering taut angle structure will be called a \emph{veering structure} for short. We note that this definition uses the orientation on $M$ in an essential way.
\end{define}

See Figure \ref{fig:veeringvertices} (ignoring the colors and labels for the moment) for the two types of vertices in $\partial \overline{M}$ 
(left-veering is left).
Notice that every fan is nonempty: for, if $v$ had an empty fan, then the two wide triangles incident to $v$ would have another common vertex $w$ that would be neither left- nor right-veering. In other words, the rightmost fan of Figure \ref{fig:veeringvertices} has the minimal number of triangles, namely 1. We will call a fan \emph{short} if it contains just one triangle, and \emph{long} otherwise.

Notice also that if $e=vv'$ is an interior edge of $M$ (not a cusp edge!), the vertices $v$ and $v'$ veer in the same direction: right- or left-veeringness is an intrinsic property of the edge $e$.

\begin{figure}
\begin{center}
\psfrag{v}{$v$}
\psfrag{t}{$T$}
\psfrag{tt}{$T'$}
\psfrag{p}{$P$}
\psfrag{pp}{$P'$}
\psfrag{q}{$Q$}
\includegraphics[width=14cm]{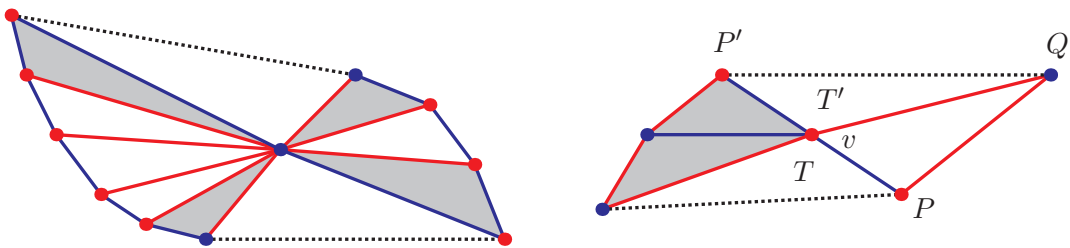}
\caption{Two vertex links in $\partial \overline{M}$ showing fans of 5, 3, 2, and one triangles (the last one a so-called \emph{short} fan). Obtuse angles are $\pi$, acute angles are $0$. The polygon on the left surrounds a left--veering vertex, and the polygon on the right a right--veering vertex. Vertices and edges receive colors (red/blue) from the veeringness condition. Hinge triangles in the fans are shaded.}
\label{fig:veeringvertices}
\end{center}
\end{figure}

Agol proved that veering structures are pleasantly common: in particular, there is a \emph{canonical} layered, veering triangulation on any pseudo-Anosov mapping torus, provided that the singularities of the invariant foliations are punctures. This canonical veering triangulation is in fact produced algorithmically by a standard (weighted) train track splitting procedure.

Shortly after Agol introduced the notion of \emph{veering}, Hodgson, Rubinstein, Segerman, and Tillmann \cite{ozzie} proved

\begin{theorem}[Theorem 1.5 of \cite{ozzie}] \label{thm:main}
Let $\tau$ be a veering ideal triangulation of $M$. Then the veering (taut angle) structure on $\tau$ can be deformed to a positive angle structure.
\end{theorem}

The proof in \cite{ozzie} is non-explicit in that it uses duality in linear programming. Namely, the linear problem of finding positive $(\theta_j)$ has a so-called \emph{dual} linear problem, which by work of Rivin \cite{rivin:combinatorial}, Kang--Rubinstein \cite{kang-rubinstein:taut-angle}, Luo--Tillmann \cite{luo-tillmann}, and others can be reduced to checking the absence of certain types of normal surfaces in the triangulation of $M$. Our aim in this paper is to give a constructive proof of Theorem \ref{thm:main},
which provides effective information about the angles.
 In fact, we show the following.

\begin{theorem}\label{thm:bound}
Let $\tau$ be a veering ideal triangulation of $M$. 
Then $\tau$ admits a positive angle structure $(\theta)$, whose angles satisfy
$$\theta_i \: \geq \:  \frac{\pi}{12 \,  d_{\max}^2} \: \geq \:  \frac{\pi}{12 \,  (e_{\max} - 3)^2},$$
where $d_{\max}$ is the maximum number of triangles in a fan of any vertex $v \in \bdy \overline{M}$, and $e_{\max} $ the largest degree of an edge of $\tau$.
\end{theorem}
 The quadratic dependence on $d_{\max}$ or $e_{\max}$ is sharp. On the other hand, the constant $\pi/12$ is very likely not sharp. See Remark \ref{rem:ptorus-bound} for more details.

Knowing explicit angles for the triangulation has several benefits. First, it is satisfying to know how to find actual examples, even though linear optimization algorithms can find them efficiently on a computer as soon as no obstruction exists. 

Second, the lower bound of Theorem \ref{thm:bound} 
is useful for algorithms in computational topology (see \emph{e.g.}\ \cite{cooper-tillmann, lackenby:heegaard-alg}). This is because known lower bounds on the angles dramatically reduce the search time required to enumerate normal surfaces of a particular genus.

Third, our construction endows every cusp with a canonical \emph{slope} attached to the veering structure. This invariant seems not to have been pointed out before.

Fourth, our construction produces positive angle structures with zero angular holonomy. That is, the turning angle about every embedded essential curve in $\partial \overline{M}$ is zero;  see Section \ref{sec:holonomies} for definitions. At present it seems to be unknown whether the existence of positive angle structures implies the existence of positive angle structures with vanishing holonomy.

Fifth, we can play further with this idea and exhibit large values of the holonomy (see Section \ref{sec:holonomies}). These values are not always the largest possible, but they still give interesting lower bounds on holonomy that might be useful for Dehn surgery arguments. In the transverse--taut case, we can also produce by deformation some taut angle structures distinct from the initial one (and usually non-veering).

Sixth, an ambitious goal is to find hyperbolic shapes on the tetrahedra $\Delta_i$ that glue up coherently to give the hyperbolic metric on $M$. By the Casson--Rivin program  \cite{rivin:volume, fg:angled-survey}, this amounts to finding a critical point of the \emph{volume functional} $\vol: \ang(\tau) \to \RR$, defined on the space $\ang(\tau)$ of positive angle structures. In practice, proving the existence of a critical point of $\vol$ requires a careful
parametrization of $\ang(\tau)$ and its boundary \cite{gf:punctured-torus}. The explicit deformations described in Sections \ref{sec:rescuing} and \ref{sec:lowerbound} are a step toward this detailed parametrization. 

Recall that the layered, veering triangulations of mapping tori constructed by Agol  are canonically determined by the pseudo--Anosov monodromy \cite{agol}. The existence of a crititcal point of $\vol$ for these triangulations would show a deep interaction between combinatorics (of train tracks, say) and hyperbolic geometry. 

Finally, the combinatorial understanding provided here might help in addressing such questions as:  
\emph{ Is every veering angle structure virtually layered (\emph{i.e.}\ a finite quotient of a layered structure)? Is there a universal bound on the volume of the union of the tetrahedra belonging to a fan?}

\subsection{Organization}
This paper is organized as follows. In Section \ref{sec:observations}, we reformulate Definition \ref{def:veering} in terms of a coloring of the edges, and explore a number of consequences for the triangulation of $\bdy \overline{M}$. In Section \ref{sec:lead-trail}, we recall the notion of \emph{leading--trailing deformations} of a generalized angle structure, naturally associated to closed curves on  $\bdy \overline{M}$. In Section \ref{sec:rescuing}, we use leading--trailing deformations to unflatten a veering taut angle structure on $\tau$ into a positive angle structure, proving Theorem \ref{thm:main}. In Section \ref{sec:lowerbound}, we analyze how far the angles can be unflattened, and prove Theorem \ref{thm:bound}.  Finally, in Section \ref{sec:holonomies}, we explore the possible holonomies of the angle structures on a veering triangulation.

\subsection{Acknowledgements}
The majority of the arguments presented in this paper were discovered during the first author's visit to Universit\'e de Lille in April 2009. At the time, Agol's definition of veering triangulations had not yet been formulated, but in retrospect we were actually studying a special case: namely, 
layered triangulations of ``Penner--Fathi'' mapping tori defined by alternating Dehn twists \cite{fathi}. 
We are grateful to Universit\'e de Lille 1  for its support and hospitality during this visit.

The appearance of \cite{agol} and \cite{ozzie} prompted us to reconsider our construction, and realize that it actually works for all veering structures.  We thank Ian Agol, Craig Hodgson, Hyam Rubinstein, Henry Segerman, and Stephan Tillmann for their stimulating ideas as well as remarks on the draft.

\section{Observations from the cusp}\label{sec:observations}

Agol's definition of veering structures can be reformulated in the following way. Note that Definition \ref{def:veering} implies that the edges of $M$ can be partitioned into two families. We color
the right-veering edges  \emph{red} and the left-veering edges  \emph{blue}.
Hodgson, Rubinstein, Segerman, and Tillmann showed the following characterization of \emph{veering}:

\begin{lemma}[Proposition 1.4 of \cite{ozzie}]\label{ozzie-veering}
A taut angle structure is veering if and only if every tetrahedron $\Delta_i$ can be sent by an orientation-preserving diffeomorphism to the one depicted in Figure \ref{fig:onetetrahedron}:
namely a thickening of the unit square, with 
\begin{itemize}
\item the $\pi$-angles on the diagonals,
\item the rising diagonal in front,
\item the vertical (thin) edges blue, and 
\item the horizontal (thin) edges red.
\end{itemize}
\end{lemma}

\begin{figure}
\begin{center}
\psfrag{o}{$0$}
\psfrag{p}{$\pi$}
\psfrag{A}{$A$}
\psfrag{B}{$B$}
\psfrag{C}{$C$}
\psfrag{D}{$D$}
\psfrag{S}{$\begin{array}{c}\text{Seen from} \\ A,B,C \text{ or } D\end{array}$}
\psfrag{H}{Hinges}
\psfrag{N}{Non-hinges}
\includegraphics[width=13cm]{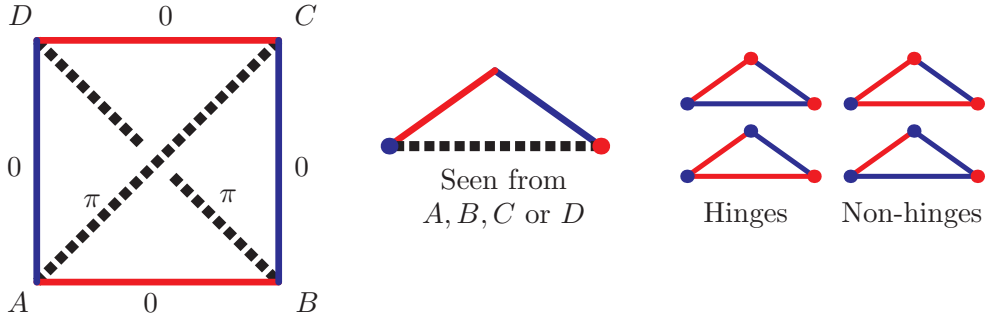}
\caption{A colored tetrahedron (left) and its cusp view (middle). Given a triangle in the cusp view, the colors of its base and top vertex determine its type (hinge or non-hinge). Throughout this paper, the lighter color (printed in greyscale) is red, and the darker color is blue.}
\label{fig:onetetrahedron}
\end{center}
\end{figure}

Note that in this characterization of veeringness, the colors of the two wide diagonals (black in Figure \ref{fig:onetetrahedron}) may be anything: \emph{i.e.}\ they are determined by the adjacent tetrahedra. This induces a partition of the tetrahedra of $M$ into \emph{hinges}, whose diagonals (\emph{i.e.}\ wide angles) bear different colors, and \emph{non-hinges}, whose diagonals are the same color. We are importing this terminology from \cite{gf:punctured-torus}.

Next, we can color all the edges and vertices of the tiling of $\partial \overline{M}$ in red and blue: the color of a vertex $v$ of $\partial \overline{M}$ is the color of the edge in $M$ incident to $v$, while the color of an edge $e \subset \bdy M$ is the color of the \emph{base} of the  ideal triangle in $M$ whose tip was truncated to yield $e$. A consequence of the red/blue characterization of veeringness is that for each triangle $T$ of $\partial \overline{M}$, if we draw the $\pi$-angle on top,
 then the left edge and right vertex are red, while the right edge and left vertex are blue. The bottom edge and top vertex could be any color (determined by some adjacent triangle $T'$); these two colors disagree if and only if $T$ is a truncation of a hinge tetrahedron (called a hinge triangle). Note that hingeness is an inherent property of the tetrahedron, inherited by all four of its boundary triangles. See Figure \ref{fig:onetetrahedron}, right.

We can now revisit Figure \ref{fig:veeringvertices} and apply colors according to the rule above. Notice that 
all vertices and edges in the figure receive a determined color, except for the bases of the two triangles that have wide angles at $v$. We can next make a series of observations:




\begin{observation}\label{obs:two-neighbors}
The vertex $v$ is connected to precisely two vertices $P,P'$ with the same color as $v$ (belonging to the two wide triangles incident to $v$). In the cyclic order for the neighbors of $v$, these two vertices are not consecutive, because each fan is nonempty. Therefore, if $v$ is for example red, then $v$ has blue neighbors on both sides (in both fans).
\end{observation}


\begin{observation}
As a consequence, if we draw all edges in $\partial \overline{M}$ that connect two vertices of the same color (incidentally, such an edge is always of the other color: check Figure \ref{fig:onetetrahedron}, right), then this defines a system of disjoint curves $\gamma_1, \gamma_2, \dots$ on $\partial \overline{M}$, passing through all vertices. We claim that no curve $\gamma_i$ can bound a disk in $\partial \overline{M}$. For, suppose without loss of generality that $\gamma_1$ has red vertices and bounds an \emph{innermost} disk. By observation \ref{obs:two-neighbors}, any vertex $v$ on this curve must have blue neighbors on both sides of $\gamma_1$. Hence there is a blue vertex inside $\gamma_1$ that belongs to some $\gamma_i$, contradicting the assumption that $\gamma_1$ was innermost.


Therefore the complement of the union of the curves $\gamma_i$ is a union of annuli $\lad_i$, because each component of $\partial \overline{M}$ has Euler characteristic $0$. The number of parallel curves $\gamma_i$ inside each torus component $T$ of $\partial \overline{M}$ is even, because their colors alternate. Note that the \emph{slope} of the $\gamma_i$ in $\mathbb{P}H_1(T, \mathbb{Z})\simeq \mathbb{Q} \cup \{ \infty \} $ is an invariant of the veering structure. See Observation \ref{obs:slope-meaning} for the meaning of this slope in the context of layered triangulations of mapping tori.
\end{observation}


\begin{observation}
Inside each annulus $\lad_i$, every edge connects one boundary component of $\lad_i$ to the other (indeed the edge has ends of distinct colors because it does not belong to the $\gamma_i$). So $\lad_i$ has the structure of a \emph{ladder}, with two \emph{ladderpoles} $\gamma_i, \gamma_{i+1}$ (red and blue) connected by many \emph{rungs}. Two consecutive rungs always have a common endpoint, and each triangle of $\lad_i$ is bounded by two rungs and one ladderpole segment (see Figure \ref{fig:cuspview} in anticipation).

The combinatorics of the rungs inside each annulus $\lad_i$ could be expressed as a cyclic sequence of ``rights'' and ``lefts,'' echoing the situation with punctured torus bundles \cite{gf:punctured-torus}. However, in our more general setting the sequences for distinct annuli are generally unrelated.
\end{observation}


\begin{figure}[h!]
\begin{center}
\psfrag{p}{$w$}
\psfrag{e}{$e$}
\psfrag{t}{$T$}
\psfrag{tt}{$T'$}
\psfrag{o}{$0$}
\psfrag{g}{$\gamma$}
\psfrag{L}{Ladder $\mathcal{L}$}
\psfrag{LL}{Ladder $\mathcal{L}'$}
\includegraphics[width=12.5cm]{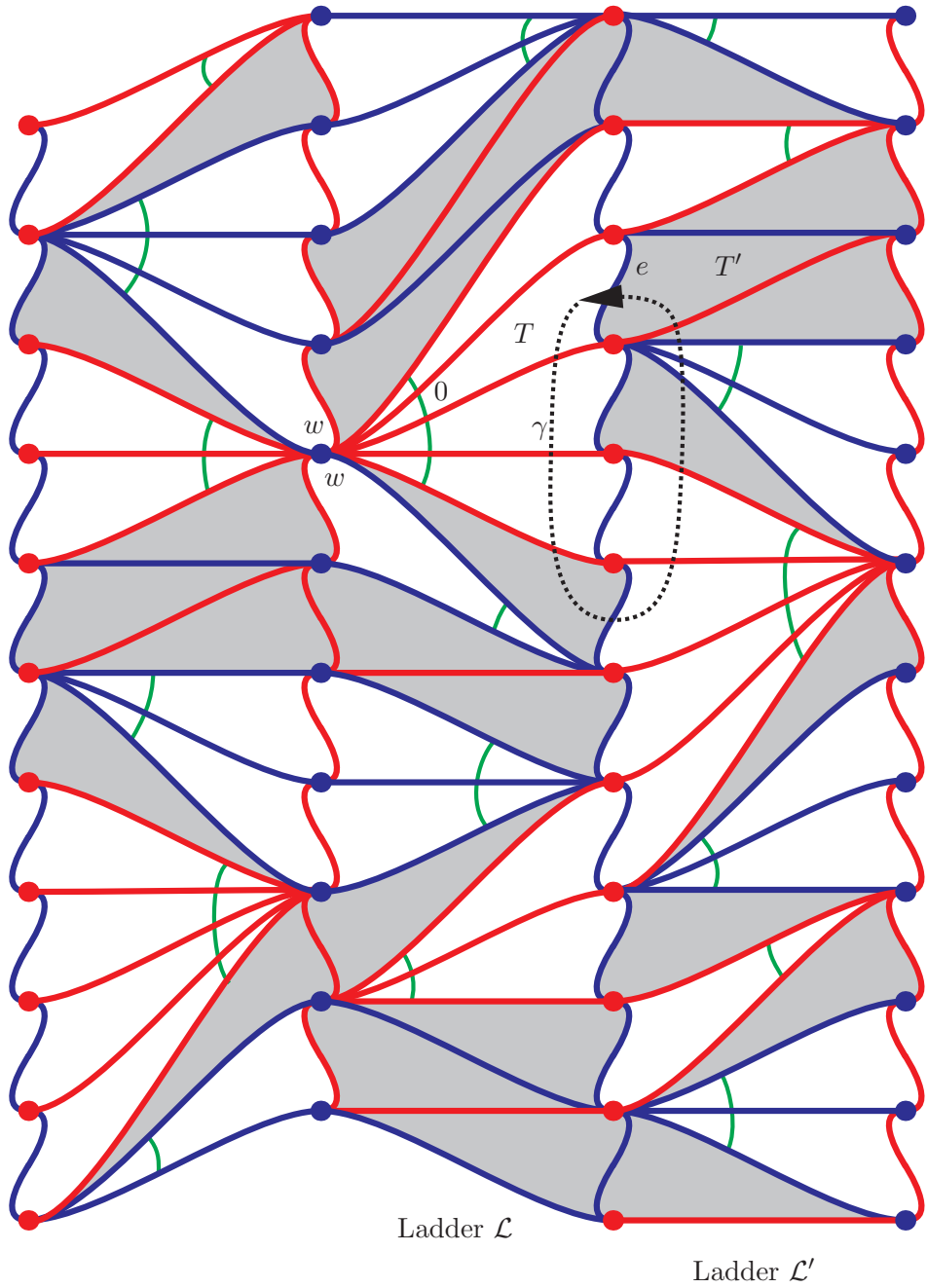}
\caption{View of $\partial \overline{M}$ with all colors. Ladderpoles are vertical, hinge triangles are shaded, and wide vs. thin angles are distinguished by a smoothing convention at each vertex. Ladderpoles may have different lengths, as the cutouts at the top and bottom suggest. Flat angles (after Lemma \ref{lemma:hinge-rescue} of the ``rescuing algorithm'') are marked in green. In Lemma \ref{lemma:nonhinge-rescue}, the curve $\gamma$ rescues the flat angle marked $0$, via the deformation $D^\gamma$.}
\label{fig:cuspview}
\end{center}
\end{figure}

\begin{observation}
For $v$ a red vertex, the two $\pi$-angles at $v$ are \emph{clockwise} just after the ladderpole through $v$, for the cyclic order on the link of $v$. For $v$ a blue vertex, the $\pi$-angles are \emph{counterclockwise} just after the ladderpole. The rest of the space on each side of the ladderpole is occupied by a full fan of $v$. This is visible in Figure \ref{fig:veeringvertices}.
\end{observation}


\begin{observation}
We can examine Figure \ref{fig:veeringvertices} to determine which triangles of the fans of $v$ are hinges. (Ignore the two triangles outside the fans of $v$, \emph{i.e.}\ the ones having $\pi$-angles at $v$: they belong to other fans. For example the triangle $T'$ belongs to a fan of $P'$ and a fan of $Q$.) A triangle is hinge if and only if its vertex at the wide angle has a color different from the opposite edge, so examination yields the following:
\begin{itemize}
\item In a long fan (of 2 or more triangles), only the first and last triangle are hinges.
\item In a short fan, the (single) triangle is not a hinge.
\end{itemize}
Hinge triangles are shaded in Figure \ref{fig:veeringvertices}.
\end{observation}


\begin{observation}
A triangle of a \emph{short} fan of a vertex $v$ (such as $vPQ$ in Figure \ref{fig:veeringvertices}) belongs to a \emph{long} fan of the vertex at its other thin corner (here $Q$), since the angles $\widehat{vQP}$ and $\widehat{vQP'}$ are both thin. (In fact, the fan at $Q$ is not just long but has length $\geq 3$, since the triangle $vPQ$, being non-hinge, cannot be the first nor the last triangle of its fan at $Q$.)

This implies in particular that \emph{there are} hinge triangles, and in fact that there are some inside every annulus $\lad_i$ of every boundary torus of $\overline{M}$.

Therefore, every hinge triangle belongs to precisely two long fans (while every non-hinge triangle belongs to precisely one long fan). Since every long fan also contains precisely two hinge triangles, we can jump from one long fan to the next according to the following scheme:
$$\begin{array}{ccccccccc}
\text{long fan}&&&&\text{long fan}&&&&\text{long fan}   \\
& \nwarrow && \nearrow && \nwarrow && \nearrow &\\
&&\text{hinge triangle}&&&&\text{hinge triangle}&&
\end{array}
$$
where arrows denote inclusion. By construction, this sequence of fans follows one of the ladders $\lad_i$.
\end{observation}


All the above observations are summarized in Figure \ref{fig:cuspview}, a view of the triangulation of $\partial \overline{M}$, also called the ``cusp view''. Ladderpoles are vertical, hinge triangles are shaded, and we use a train track-like smoothing convention to emphasize which angles near any given vertex are the wide angles (look \emph{e.g.}\ for the two $w$-labels for ``wide'').

We close this section with two observations that hold true under additional hypotheses.

\begin{observation}\label{obs:transverse}
Suppose that the veering triangulation $\tau$ is transverse--taut. Recall this means there is a transverse orientation on all the faces (\emph{e.g.}\ towards the reader in the first panel of Figure \ref{fig:onetetrahedron}) that is consistent over all tetrahedra. This orientation induces an ``upward'' orientation along the ladders in each cusp triangulation. Inside each cusp, ladders of two types alternate: in one type, the wide angle of each triangle is always \emph{above} the base (as in ladder $\mathcal{L}$ in Figure \ref{fig:cuspview}); in the other type, the wide angle is always \emph{below} the base (as in ladder  $\mathcal{L}'$). We call the first type of ladder \emph{ascending}, and the second type \emph{descending}. 

If the transverse orientation of Figure \ref{fig:onetetrahedron} points toward the reader, the truncated vertices at $A$ and $C$ belong to ascending ladders, while the truncated vertices at $B$ and $D$ belong to descending ladders. In particular, each tetrahedron has two vertices in each type of ladder; this fact will be crucial in Sections \ref{sec:lowerbound} and \ref{sec:holonomies}.
\end{observation}

\begin{observation}\label{obs:slope-meaning}
Suppose that the veering triangulation $\tau$ comes from Agol's construction of  a layered triangulation of a mapping torus with monodromy $\phi$, with punctures at the singularities of the $\phi$--invariant foliations $\mathcal{F}, \mathcal{F}'$. Then 
%
the germs of the singular leaves of $\mathcal{F}, \mathcal{F}'$ incident to a puncture of the fiber
define a slope in each cusp of $M$. One can prove that these slopes are the same as the slopes of the ladderpoles $\gamma_i$. 
At each cusp of $M$, the number of prongs of $\mathcal{F}$ (or $\mathcal{F}'$) is equal to the number of pairs of ladders, multiplied by the intersection number of the
ladder slope with the fiber slope.
\end{observation}

\section{Leading--trailing deformations}\label{sec:lead-trail}

Let $GAS(\tau)$ be the set of generalized angle structures on a triangulation $\tau$, where the angles $\theta_j$ can take any value in $\RR$. 
In this section, we exhibit a spanning set for the tangent space $T_p GAS(\tau)$. All the facts in this section are proved in \cite[Section 4]{fg:angled-survey}, and we point to that paper for more detail.

\begin{define}\label{def:lead-trail}
A \emph{normal curve} on a component of $\partial \overline{M}$ is an embedded, oriented closed curve $\gamma$ transverse to the $1$--skeleton of the triangulation of $\partial \overline{M}$, such that $\gamma$ enters and exits each triangle through different edges. 

For each edge $e$ crossed by a normal curve $\gamma$, there is a real variable $\theta_{j_e}$ associated to the angle opposite $e$ in the triangle \emph{entered} by $\gamma$, and a real variable $\theta_{j'_e}$ associated to the angle opposite $e$ in the triangle \emph{left} by $\gamma$. Let $\eps_j$ be the $j$-th basis vector of $\mathbb{R}^{3n}$. Then the \emph{leading--trailing deformation} associated to $\gamma$ is the 
vector
$$D^\gamma := \sum_e (\eps_{j_e}-\eps_{j'_e})~.$$
Here each edge $e$ can appear many times in the sum (as often as it is crossed by $\gamma$) and the indices $j_e, j'_e$ may be swapped according to the direction in which $\gamma$ crosses $e$.
\end{define}

\begin{lemma}[Lemma 4.5 of \cite{fg:angled-survey}]
The vector $D^\gamma$ is tangent to $GAS(\tau)$. That is, if $(\theta)=(\theta_j)_{1\leq j \leq 3n}$ is a generalized angle structure, then $(\theta)+tD^\gamma$ is also one, for every real $t$.
\end{lemma}

In other words, given an angle structure $(\theta)$, one may increase (resp. decrease) by $t$ all the angles opposite edges crossed by $\gamma$, according to the direction of crossing (with multiplicity). An equivalent and sometimes useful way of seeing the deformation $D^\gamma$ is as follows: whenever $\gamma$ traverses a triangle $abc$ by entering through $ab$ and leaving through $bc$, increase the angle at $c$ and decrease the angle at $a$. See Figure \ref{fig:perturbation}.

\begin{remark}
Suppose that $\gamma$ crosses a boundary triangle of a (truncated) tetrahedron $\Delta$. Then the vector $D^\gamma$ deforms the dihedral angles of four edges in $\Delta$, which are adjacent to all four boundary triangles of $\Delta$. As a result, the deformation $D^\gamma$ can affect the shapes of boundary triangles that do not intersect $\gamma$, including truncation triangles that belong to completely different components of $\partial \overline{M}$.
\end{remark}


\begin{remark}\label{remark:edge-def}
Let $\epsilon$ be an edge in the interior of $M$ that connects two vertices $v,v' \in \partial \overline{M}$. Then the clockwise loops $\gamma_v$ and $\gamma_{v'}$ that encircle $v$ and $v'$, respectively,  induce \emph{identical} deformations $D^{\gamma_v}=D^{\gamma_{v'}} \in T_p GAS(\tau)$. We call this common deformation  $D^\epsilon$. This is in contrast to other curves $\gamma$, each of which lives on a particular boundary component of $\partial \overline{M}$.
\end{remark}


The deformation $D^{\gamma_v}$ is illustrated in Figure \ref{fig:perturbation}, right. A comparison of this picture with Figure \ref{fig:veeringvertices} should convince the reader that $D^{\gamma_v}$ (resp. $-D^{\gamma_v}$) can be used to unflatten the triangles in the two fans of a blue (resp. red) vertex $v$, although it might place a negative angle at the thin vertices of the two wide triangles incident to $v$. The next section develops this idea.

\begin{figure}
\begin{center}
\psfrag{p}{$+$}
\psfrag{m}{$-$}
\psfrag{j}{$j_e$}
\psfrag{jj}{$j'_e$}
\psfrag{e}{$e$}
\psfrag{g}{$\gamma$}
\psfrag{gv}{$\gamma_v$}
\psfrag{v}{$v$}
\includegraphics[width=12cm]{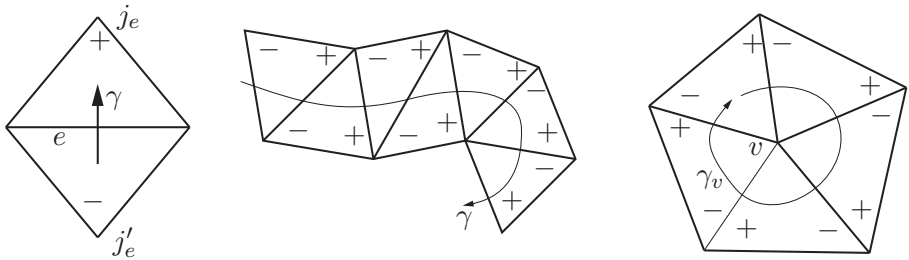}
\caption{Various (segments of) curves $\gamma$ and the associated perturbations $D^\gamma$, expressed via signs $+$ and $-$.}
\label{fig:perturbation}
\end{center}
\end{figure}

\section{Rescuing the zeros}\label{sec:rescuing}


In this section, we describe an algorithm to deform any veering structure to a positive angle structure, by applying various deformations $D^\gamma$. This algorithm will \emph{rescue} (that is, unflatten) the zero angles of the tetrahedra, one at a time. The hinge tetrahedra are rescued in Lemma \ref{lemma:hinge-rescue}, and the non-hinge tetrahedra in Lemma \ref{lemma:nonhinge-rescue}. Together, these lemmas imply Theorem \ref{thm:main}.

%

For each edge $\epsilon$ in the interior of $M$, let $D^\epsilon$ be the clockwise deformation about $\epsilon$, as described in Remark \ref{remark:edge-def} and shown in Figure \ref{fig:perturbation}, right. Define
\begin{equation}\label{eq:hinge-def}
D := \sum_{\epsilon \text{ blue}} D^\epsilon - \sum_{\epsilon \text{ red}} D^\epsilon.
\end{equation}

\begin{lemma}\label{lemma:hinge-rescue}
Let $(\theta)$ be a veering taut angle structure on $M$. Choose any $t\in (0, \pi/4)$, and deform $(\theta)$ to the generalized angle structure $(\theta') = (\theta) + tD$, for the deformation $D$ in \eqref{eq:hinge-def}. Then $(\theta')$ has the following properties:
\begin{enumerate}
\item\label{hinges-saved} Every hinge tetrahedron (or triangle) has positive angles.
\item\label{nonhinge-partial} Every non-hinge tetrahedron has non-negative angles.
\item\label{0-ladderpole}  All $0$--angles are opposite ladderpole segments in Figure \ref{fig:cuspview}.
\end{enumerate}
\end{lemma}

In fact, after applying the deformation $tD$, every non-hinge triangle will have precisely one vanishing angle, namely the one opposite the ladderpole\footnote{It is worth describing what happens in the special case of punctured torus bundles. After applying Lemma \ref{lemma:hinge-rescue}, in the notation and terminology of \cite{gf:punctured-torus}  we have ``$w_i\equiv 2t$'' for all $i$, so all ``hinge conditions'' are satisfied and one only needs to relax the ``concavity conditions''.}. The vanishing angles are marked green in Figure \ref{fig:cuspview}.

\begin{proof}
Fix a tetrahedron $\Delta$, hinge or non-hinge. It has two blue edges with $0$ angles (call these $e_1, e_2$), two red edges with $0$ angles (call these $e_3, e_4$), and two diagonals with $\pi$ angles (call these $e_5, e_6$). We will treat the deformations $D^{e_j}$ one at a time, as though $e_1, \ldots, e_6$ are distinct edges in $M$. For, if some of these edges (for concreteness, $e_1$ and $e_2$) are identified to the same edge $e \subset M$, the deformation of $D^e$ affects $\Delta$ in two different ways, one that corresponds to $e_1$ and the other that corresponds to $e_2$. Thus the cumulative effect of the deformations $D^{e_j}$ will be the same, regardless of whether 
some edges of $\Delta$ are identified.

First, consider the four thin edges of $\Delta$.  The clockwise deformation $t D^{e_1}$  decreases the $\pi$--angles by $t$ and increases the red $0$--angles by $t$. The effect of $t D^{e_2}$ is exactly the same. In a similar fashion, each of the counterclockwise deformations $-t D^{e_3}$ and $-tD^{e_4}$ decreases the $\pi$--angles by $t$ and increases the blue $0$--angles by $t$.
We conclude that after performing the deformations along the four thin edges of $\Delta_i$, the dihedral angles of the tetrahedron become 
\begin{equation}\label{eq:hinge-angles}
(\pi-4t, 2t, 2t).
\end{equation}

Next, consider the diagonals $e_5, e_6$ of $\Delta$. Observe that deformations $D^{e_5}$ and $D^{e_6}$ along opposite edges of $\Delta$ have precisely the same effect on $\Delta$. If $\Delta$ is a hinge tetrahedron, then its two diagonals have opposite colors, and the effect of $tD^{e_5} -t D^{e_6}$ cancels out completely. Thus $\Delta$ ends up with angles $(\pi-4t, 2t, 2t),$ which are always positive for $t \in (0, \pi/4)$. This proves \eqref{hinges-saved}.

If $\Delta$ is not a hinge and both diagonals are red, then each of the deformations $tD^{e_5}$ and $tD^{e_6}$ increases the thin red angles by $t$, and decreases the thin blue angles by $t$. Thus, after performing the deformations along both $e_5$ and $e_6$, the angles of $\Delta$ become 
\begin{equation}\label{eq:nonhinge-angles}
(\pi-4t, 4t, 0), 
\end{equation}
where the blue angles are $0$. 
%
Observe in Figure \ref{fig:cuspview} (or Figure \ref{fig:veeringvertices}) that in a non-hinge triangle coming from a tetrahedron with red diagonals, the ladderpole connects red vertices and is opposite the blue thin angle. Thus $\Delta$ will have non-negative angles everywhere, and $0$ precisely at the angles opposite the ladderpole segments.

If $\Delta$ has blue diagonals, the effect is exactly symmetric. The deformations along $e_5$ and $e_6$ will make the angles of $\Delta$ become $(\pi-4t, 0, 4t)$, with zeros at the thin red angles  opposite the ladderpole segments. This proves \eqref{nonhinge-partial} and \eqref{0-ladderpole}.
\end{proof}

Next, we rescue the non-hinge tetrahedra, using the following inductive procedure.

\begin{lemma}\label{lemma:nonhinge-rescue}
Suppose that the tetrahedra $\Delta_1, \ldots, \Delta_n$ are endowed with a generalized angle structure $(\theta) = (\theta_1, \ldots, \theta_{3n})$, such that these angle assignments satisfy properties \eqref{hinges-saved}, \eqref{nonhinge-partial}, and \eqref{0-ladderpole} of Lemma \ref{lemma:hinge-rescue}. Suppose as well that some angle $\theta_j$ is $0$. 

Then there is a homotopically trivial closed curve $\gamma \subset \bdy M$, such that for small $t > 0$ the structure $(\theta') = (\theta) + t D^{\gamma}$ again satisfies properties \eqref{hinges-saved}, \eqref{nonhinge-partial}, and \eqref{0-ladderpole} of Lemma \ref{lemma:hinge-rescue}, but has strictly fewer vanishing angles than $(\theta)$.
\end{lemma}

\begin{proof}
Call a tetrahedron \emph{degenerate} if it has an angle labeled $0$, and \emph{non-degenerate} otherwise. Similarly, call a boundary triangle \emph{flat} if it has an angle of $0$, and \emph{non-flat} otherwise.

Since $M$ is connected, there must be a degenerate tetrahedron $\Delta$ that is glued along face $f$ to a non-degenerate tetrahedron $\Delta'$. By looking at the appropriate truncated vertex of $f$ (three possible choices), we can assume that in $\partial \overline{M}$, we have the 0-angle in a triangle $T \subset \Delta$, opposite an edge $e$ shared with a non-flat triangle $T' \subset \Delta'$. By property  \eqref{0-ladderpole}, $e$ must be a ladderpole segment that forms part of the boundary between annuli $\lad$ and $\lad'$, with $T \subset \lad$ and $T' \subset \lad'$. See Figure \ref{fig:cuspview}.

Begin the curve $\gamma$ by crossing edge $e$ from $T'$ into $T$. (When crossing into $T$ opposite the vanishing angle, we instantly unflatten the angle of $0$, by definition of $D^\gamma$.) Next, travel vertically through the fan of $T$ in annulus $\lad$, until coming to a hinge triangle that shares an edge with $\lad'$. (Note that a hinge triangle at the top or bottom of the fan of $T$ must be adjacent to $\lad'$; inspection in Figure \ref{fig:cuspview} shows that when traversing the fan of $\lad$, the curve $\gamma$ must exit each triangle of the fan across the edge connecting the two \emph{thin} angles.) Cross back into $\lad'$ at the hinge, and travel in the opposite direction, until $\gamma$ returns to triangle $T'$ and closes up.

Note that when $\gamma$ travels vertically and crosses the rungs of a ladder, it never leaves a triangle opposite an angle of $0$, hence all angles decreased by $D^\gamma$ are strictly positive. In addition, $\gamma$ leaves $\lad$ through a hinge triangle that has positive angles by \eqref{hinges-saved}, and leaves $\lad'$ through a triangle $T'$ that has positive angles by induction hypothesis. Thus $D^\gamma$ never decreases an angle of $0$, and preserves properties \eqref{hinges-saved}, \eqref{nonhinge-partial}, and \eqref{0-ladderpole} while ``rescuing'' at least one tetrahedron.
\end{proof}

We may repeatedly apply Lemma \ref{lemma:nonhinge-rescue} until all angles are positive. This completes our  constructive proof of Theorem  \ref{thm:main}. \qed

\section{A lower bound on the smallest angle} \label{sec:lowerbound}

In this section, we prove Theorem \ref{thm:bound}. We will first prove the theorem under the additional hypothesis that the triangulation $\tau$ is transverse--taut, and then apply a covering argument to extend the result to all veering triangulations.

\begin{proposition}\label{prop:transverse-bound}
Let $\tau$ be a transverse--taut, veering ideal triangulation of $M$. 
Then $\tau$ admits a positive angle structure $(\theta)$, whose angles satisfy
\begin{equation}\label{eq:angle-bound}
\theta_i \: \geq \:  \frac{\pi}{12 \,  d_{\max}^2} \: \geq \:  \frac{\pi}{12 \,  (e_{\max} - 3)^2},
\end{equation}
where $d_{\max}$ is the maximum number of triangles in a fan of any vertex $v \in \bdy \overline{M}$, and $e_{\max} $ the largest degree of an edge of $\tau$.
\end{proposition}

\begin{remark}\label{rem:ptorus-bound}
The quadratic dependence on $d_{\max}$ is optimal, even in the special case of punctured torus bundles. In the notation of \cite{gf:punctured-torus}, a non-hinge tetrahedron in the $i$-th spot of a syllable $LR^nL$ will have an angle of the form 
$$x_i = 2w_i - (w_{i-1} + w_{i+1}),$$ 
where all parameters must satisfy $w_j \in (0, \pi/2)$. The requirement $x_i > 0$ implies that the sequence of parameters $w_{i-1}, w_i, w_{i+1}$ is concave, hence is called the \emph{concavity condition}. Now, a fan of length $(n+1)$ imposes a concave sequence $w_0, \ldots, w_n$ that must stay at most distance $\pi/2$ above the line segment from $(0, w_0)$ to $(n, w_n)$. Summing by parts twice, we see that the range condition $w_j \in (0, \pi/2)$ cannot be satisfied if $x_i \geq 4\pi/n^2$ for all $i = 1, \ldots, n-1$. 
This indicates that the quadratic behavior is sharp but the constant $\pi/12$ is probably far from sharp.

In fact, \cite[Proposition 10.1]{gf:punctured-torus} implies that when the monodromy is $L^n R^n$ and $(\theta)$ is the \emph{geometric} structure on the triangulation (\emph{i.e.}, the unique angle structure under which the tetrahedra glue up to give the hyperbolic metric on $M$), the smallest angle will be approximately $4\pi^2/n^3$.
\end{remark}

\begin{proof}[Proof of Proposition \ref{prop:transverse-bound}]
The proof will use the same deformations $D^\gamma$ as in Lemmas \ref{lemma:hinge-rescue} and \ref{lemma:nonhinge-rescue}, with careful choices of coefficient $t$.

Given a transverse orientation on the faces of $\tau$ (say, towards the reader in the first panel of Figure \ref{fig:onetetrahedron}), Observation \ref{obs:transverse} implies that each tetrahedron has exactly two tips in ascending ladders, and two tips in descending ladders.
Consider the collection $\mathcal{T}$ of all cusp triangles belonging to ascending ladders: $\mathcal{T}$ has cardinality $2n$ if there are $n$ tetrahedra. The set $\mathcal{T}$ is naturally endowed with a fixed-point-free involution $\sigma$ taking each triangle to the only other triangle in $\mathcal{T}$ that belongs to the same tetrahedron.

Let $T$ be a triangle in $\mathcal{T}$, belonging to some ascending ladder $\mathcal{L}$. Define the \emph{height} of $T$, denoted $H(T)$, to be the length of the shortest path \emph{down} the ladder $\mathcal{L}$ that connects $T$ to a hinge triangle: for example, hinge triangles have height $0$; their nonhinge neighbors \emph{immediately above} in the ascending ladder have height $1$, and so on. There is no obvious \emph{a priori} relationship between $H(T)$ and $H(\sigma(T))$, except that they can only vanish simultaneously (when $T, \sigma(T)$ belong to a hinge tetrahedron).

Further, as in the proof of Lemma \ref{lemma:nonhinge-rescue}, define $\gamma_T$ to be the oriented path that enters $T$ through the ladderpole edge (leaving another ladder $\mathcal{L}'\neq \mathcal{L}$), travels $H(T)$ rungs down the ladder $\mathcal{L}$ until it reaches a hinge triangle, then crosses back into $\mathcal{L}'$ and 
travels across the rungs of $\mathcal{L}'$ to close up. For a hinge triangle $T$, by convention we define $\gamma_T$ to be the trivial (\emph{i.e.}\ empty) path.

Consider the generalized angle structure $(\theta')$ given by Lemma \ref{lemma:hinge-rescue} with $t=\frac{\pi}{6}$. By equations \eqref{eq:hinge-angles} and \eqref{eq:nonhinge-angles}, the angles of $(\theta')$ have the following properties:
\begin{itemize}
\item Hinge triangles are equilateral, with angles $(\frac{\pi}{3},\frac{\pi}{3},\frac{\pi}{3})$;
\item Nonhinge triangles have angles $(\frac{\pi}{3},\frac{2\pi}{3},0)$, with $\frac{\pi}{3}$ at the wide angle and $0$ belonging to a long fan.
\end{itemize}

Define $d$ to be the number of \emph{non-hinge} triangles in the longest fan in an ascending ladder.
Equivalently, $d=\max_{T \in \mathcal{T}} H(T)$ is the largest number of adjacent green arcs in an ascending ladder in Figure \ref{fig:cuspview}. Note that by the definition of $d_{\max}$, one has $d + 2 = d_{\max}$. (The inequality $d+2\leq d_{\max}$ is immediate, since $d_{\max}$ includes hinge triangles and counts both ascending and descending ladders. Equality holds because to each fan in a descending ladder at an endpoint of an edge $e$ of $M$ corresponds a fan of the same length in an ascending ladder at the other endpoint of $e$.)
Also, by construction, every cusp triangle (whether in the ascending collection $\mathcal{T}$ or not) is crossed by at most $2d$ curves $\gamma_T$: this is because the curves $\gamma_T$ are nested by families of at most $d$ along the ladderpoles, and each ladder has $2$ poles.

If $d=0$, then all tetrahedra are hinges, $d_{\max}= 2$ and $e_{\max}=6$. Inequality \eqref{eq:angle-bound} easily holds for the equilateral hinge triangles. Thus we may suppose that  $d\geq 1$.

\begin{claim}
Assume that $d \geq 1$, and let $\kappa=\frac{\pi}{24}$. Then the angle structure 
\begin{equation}\label{eq:bounded-rescue}
 (\theta) :=(\theta')+\sum_{T \in \mathcal{T}} \frac{\kappa}{d_{\max}^2}\cdot H(\sigma(T)) D^{\gamma_T}
\end{equation}
is positive, with smallest angle equal to at least $2\kappa/ d_{\max}^2$. (We may see the sum above as being over \emph{all} triangles $T \in \mathcal{T}$, even though hinge triangles contribute $0$.)
\end{claim}

This claim clearly finishes the proof that all angles are at least $2\kappa/ d_{\max}^2$. The remaining inequality in the Proposition (relating angles to the maximum edge degree $e_{\max}$) follows from the observation that $e_{\max} \geq d_{\max} + 3$: for, $d_{\max}$ only counts the thin angles in one fan at $v$, but there must also be two wide angles, and at least one angle in the other fan. Thus all that remains is to prove the claim.


First, consider a hinge tetrahedron $\Delta$ which contributes cusp triangles $T_1, T_2 \in \mathcal{T}$ and $T_3, T_4 \notin \mathcal{T}$. Each of $T_1, \dots, T_4$ is crossed by at most $2d$ curves $\gamma_T$, and each curve carries weight at most
 $$\frac{ \kappa \cdot \max H }{ d_{\max}^2 } =\frac{\kappa  d }{d_{\max}^2}.$$ Each angle of $\Delta$ starts out with a value of $\pi/3 = 8 \kappa$ and is affected by at most $
4 \cdot 2d \cdot \kappa d / d_{\max}^2$. Thus all the $(\theta)$-angles of $\Delta$ are at least 
$$
\theta_j 
\geq 
8 \kappa - \frac{ 8 \kappa d^2 }{ d_{\max}^2 } 
= 8 \kappa \left( \frac{d_{\max}^2 - d^2 }{   d_{\max}^2}  \right)
= 8 \kappa \left( \frac{(d+2)^2 - d^2 }{   d_{\max}^2}  \right)
= 8 \kappa \left( \frac{4 d + 4 }{   d_{\max}^2}  \right)
\geq \frac{64 \kappa } {   d_{\max}^2} ~. \\
$$

Next, consider a non-hinge triangle $\Delta$. The same calculation as above applies to show that the angles of $\Delta$ that are nonzero for the angle structure $(\theta')$ are still at least $64\kappa/d_{\max}^2$ for the angle structure $(\theta)$. It only remains to deal with the flat angle of $\Delta$.

\begin{figure}
\begin{center}
\psfrag{t1}{$T_1$}
\psfrag{t2}{$T_2$}
\psfrag{t3}{$T_4$}
\psfrag{t4}{$T_3$}
\psfrag{p3}{$T'_4$}
\psfrag{p4}{$T'_3$}
\psfrag{s3}{$\sigma(T'_4)$}
\psfrag{s4}{$\sigma(T'_3)$}
\includegraphics[width=14cm]{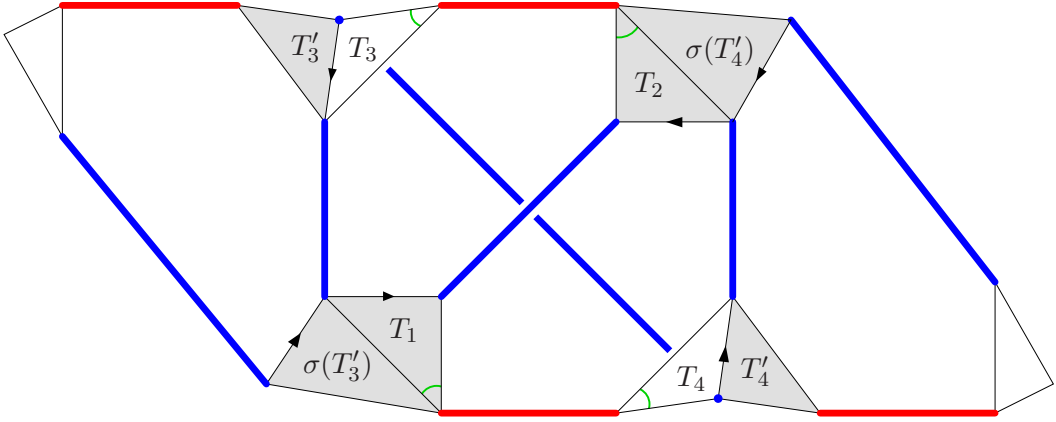}
\caption{A non-hinge tetrahedron $\Delta$ and its cusp triangles $T_1,\dots, T_4$, together with two adjacent tetrahedra. Inside each cusp, the ladderpole edges (always connecting vertices of the same color) are shown with their upward orientation induced by the transverse--taut structure. Triangles of the ascending collection $\mathcal{T}$ are shaded, and the $0$-angles of $\Delta$ for $(\theta')$ are marked in green as in Figure \ref{fig:cuspview}. For a non-hinge tetrahedron with four red edges, the figure would be reflected across a line of slope $+1$, with red and blue interchanged.}
\label{fig:fourcusps}
\end{center}
\end{figure}

We still assume that $\Delta$ has cusp triangles $T_1, T_2 \in \mathcal{T}$ and $T_3, T_4 \notin \mathcal{T}$. By definition, $\sigma$ exchanges $T_1$ and $T_2$. Let $T'_3\in \mathcal{T}$ be the neighbor of $T_3$ sharing a ladderpole segment with $T_3$, and $T'_4\in \mathcal{T}$ be the neighbor of $T_4$ sharing a ladderpole segment with $T_4$. The following deformations $D^{\gamma_T}$ all affect the 0-angle of $\Delta$:
\begin{itemize}
\item  $D^{\gamma_{T_1}}$ (positively because $\gamma_{T_1}$ \emph{enters} $T_1$ through the ladderpole);
\item  $D^{\gamma_{T_2}}$ (positively because $\gamma_{T_2}$ \emph{enters} $T_2$ through the ladderpole);
\item $D^{\gamma_{T'_3}}$ (negatively because $\gamma_{T'_3}$ \emph{leaves} $T_3$ through the ladderpole);
\item $D^{\gamma_{T'_4}}$ (negatively because $\gamma_{T'_4}$ \emph{leaves} $T_4$ through the ladderpole).
\end{itemize}
Any additional deformations $D^{\gamma_T}$ may only affect the $0$-angle of $\Delta$ positively: this occurs when $\gamma_T$ happens to enter a triangle of $\Delta$ when it crosses back into the descending ladder $\mathcal{L}'$ (in other words, when $T'_3$ or $T'_4$ is a hinge and happens to be the nearest hinge below $T$ in the ascending ladder $\mathcal{L}$).

The key observation is that $\sigma(T'_3)$ and $\sigma(T'_4)$ are the neighbors of $T_1$ and $T_2$ \emph{down} the ascending ladders: see Figure \ref{fig:fourcusps}. This implies 
\begin{equation}
H(\sigma(T'_3))=H(T_1)-1 \qquad \text{and} \qquad H(\sigma(T'_4))=H(T_2)-1. 
\end{equation}
As a result, by equation \eqref{eq:bounded-rescue}, the angle $\theta'_j = 0$ of  $\Delta$ will become
\begin{eqnarray*}
\theta_j & \geq &
 \frac{\kappa}{d_{\max}^2} \cdot \left ( 
H(\sigma(T_1)) D^{\gamma_{T_1}} + 
H(\sigma(T_2)) D^{\gamma_{T_2}} + 
H(\sigma(T'_3)) D^{\gamma_{T'_3}} + 
H(\sigma(T'_4)) D^{\gamma_{T'_4}} \right )_j \\
&\geq & \ \frac{\kappa}{d_{\max}^2}\cdot \left ( 
H(\sigma(T_1))+ H(\sigma(T_2))- H(\sigma(T'_3))- H(\sigma(T'_4))\right ) \\
&=&  \frac{\kappa}{d_{\max}^2} \cdot \left ( 
H(T_2)+ H(T_1)- [H(T_1)-1]- [H(T_2)-1] \right ) \\
& = &  \frac{2\kappa}{d_{\max}^2},
\end{eqnarray*}
completing the proof.
\end{proof}

To prove Theorem \ref{thm:bound} in general, we need to lift the angle structure $(\theta)$ to a transverse--taut double cover. The following lemma establishes the existence of such a cover.

\begin{lemma}\label{lemma:transverse-cover}
Let $\tau$ be an ideal triangulation of $M$ with a taut angle structure. If $\tau$ is not transverse--taut, then there is a double cover $N \to M$, such that the lift of $\tau$ to $N$ is transverse--taut. 
\end{lemma}

\begin{proof}
The (unoriented) transverse direction defines a line bundle $B\rightarrow M$, which is a natural subbundle of $TM$. If $B$ is orientable (\emph{i.e.}\ admits a nonzero section), then $\tau$ is tansverse-taut. If not, then $B$ lifts to a line bundle $\widetilde{B}$ over the universal cover $\widetilde{M}$ of $M$, and $\widetilde{B}$ is orientable because $\widetilde{M}$ is simply connected. The group of deck transformations of $\widetilde{M}$ which preserve the orientation of $\widetilde{B}$ has index $2$, and the corresponding double cover $N$ of $M$ satisfies the conditions.
\end{proof}

\begin{proof}[Proof of Theorem \ref{thm:bound}]
Let $\tau$ be a veering, taut triangulation of $M$. If the taut angle structure on $\tau$ is actually transverse--taut, then the proof is complete by Proposition \ref{prop:transverse-bound}. Otherwise, Lemma \ref{lemma:transverse-cover} guarantees a double cover $N \to M$, such that the lift $\tilde{\tau}$ of $\tau$ is transverse--taut. Note that tetrahedra, angles, and edge degrees all lift to finite covers. Thus $d_{\max}(M) =  d_{\max}(N)$ and $e_{\max}(M) =  e_{\max}(N)$.

By Proposition \ref{prop:transverse-bound}, the transverse--taut triangulation $\tilde{\tau}$ of $N$ admits a positive angle structure $(\theta)$, with all angles satisfying inequality \eqref{eq:angle-bound}. Let $\sigma$ be the involution of $N$ that acts as a deck transformation of the (regular) double cover $N \to M$. Then, because the polytope of positive angle structures is convex,
$$(\theta') = \frac{(\theta) + \sigma(\theta)}{2}$$ is also an angle structure on $N$, also satisfies \eqref{eq:angle-bound}, and is $\sigma$--equivariant. Projecting the angles of $(\theta')$ down to $M$ completes the proof.
\end{proof}

\section{Holonomies} \label{sec:holonomies}

\subsection{Definitions}
In this section, we explore the possible holonomies of the positive angle structures on a veering triangulation. Informally, the holonomy of a curve $\gamma$ in $\partial \overline{M}$ is its total turning angle according to the angles $\theta_j$ of the angle structure. We will show that the holonomy of $\gamma$ only depends on the homology class $[\gamma]$, and is a linear functional on homology classes. In other words, the holonomy naturally lives in the first cohomology of $\partial \overline{M}$.

\begin{define}\label{def:holonomy}
Let $V$ be a torus component of $\partial \overline{M}$, endowed with its possibly non-simplicial triangulation. Let $\gamma \subset V$ be an oriented, normal closed curve, as in Definition \ref{def:lead-trail}.  Then a component of $\gamma$ in a triangle of $V$ cuts off exactly one corner the triangle, which is either to the left or right of $\gamma$. 

\begin{figure}
\begin{center}
\psfrag{a1}{$a_1$}
\psfrag{a2}{$a_2$}
\psfrag{b1}{$b_1$}
\psfrag{b2}{$b_2$}
\psfrag{do}{$\dots$}
\psfrag{ga}{$\gamma$}
\includegraphics[width=8cm]{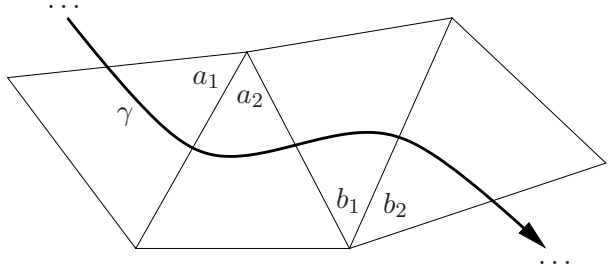}
\caption{Here $\mathcal{A}_\gamma=a_1+a_2+(\dots)$ and $\mathcal{B}_\gamma=b_1+b_2+(\dots)$.}
\label{fig:defholonomy}
\end{center}
\end{figure}

Let $(\theta)$ be a generalized angle structure on the triangulation $\tau$. Then every corner of each triangle on $V$ receives a (real-valued) ``angle'' assignment.
Let $\mathcal{A}_\gamma$ be the sum of angles to the left of $\gamma$, and $\mathcal{B}_\gamma$ the sum of angles to the right of $\gamma$. Note that each $\theta_j$ may appear several times in $\mathcal{A}_\gamma$ and $\mathcal{B}_\gamma$. Then we define the \emph{turning angle of $\gamma$} to be
\begin{equation}\label{eq:turning-angle}
t_\theta(\gamma) \: = \: \mathcal{A}_\gamma-\mathcal{B}_\gamma~. 
\end{equation}
See Figure \ref{fig:defholonomy}. Recall that every primitive homology class in $H_1(V, \ZZ)$ is represented by a simple closed curve, which can be taken to be normal after a small isotopy. If $[\gamma]$ is a nonzero, primitive homology class and $n \in \ZZ$, define the \emph{angular holonomy} of $n[\gamma] \in H_1(V, \ZZ)$ to be
\begin{equation}\label{eq:holonomy}
h_\theta(n[\gamma]) \: = \: n(\mathcal{A}_\gamma-\mathcal{B}_\gamma),
\end{equation}
for an arbitrary normal curve $\gamma$ representing $[\gamma]$.
\end{define}

Equation \eqref{eq:holonomy} begs the question of whether holonomy is well-defined. In fact, we have

\begin{proposition}\label{prop:holonomy-rep}
Let $V$ be a torus component of $\partial \overline{M}$, and $(\theta)$ a generalized angle structure on $\tau$. Let $\gamma \subset V$ be an oriented, normal, homologically non-trivial closed curve. Then the turning angle of $\gamma$ only depends on the homology class of $\gamma$, hence the holonomy $h_\theta([\gamma])$ is well-defined. Furthermore, for $a,b \in \ZZ$ and homology classes $\omega, \eta \in H_1(V, \ZZ)$, 
$$h_\theta(a \omega + b \eta) = a \cdot h_\theta(\omega) + b \cdot h_\theta(\eta).$$
In other words, $h_\theta: H_1(\bdy \overline{M}) \to \RR$ is a linear functional, hence $h_\theta \in H^1 (\bdy \overline{M}, \RR)$.
\end{proposition}

To appreciate the non-triviality of this claim, observe that the trivial homology class is represented by a curve $\gamma$ encircling a single vertex of $V$  (see Figure \ref{fig:perturbation}). All the angles cut off by $\gamma$ lie to one side of the curve, and by Definition \ref{def:anglestructure} these angles sum to $2\pi$.  Thus the turning angle of this trivial curve is $\pm 2\pi$, depending on the orientation of $\gamma$, rather than $0$. 

We note that angular holonomy is also a restriction of a well-known definition of complex--valued holonomy, which keeps track of complex--valued shape parameters instead of just
real--valued angles. (See Thurston's notes \cite[Chapter 4]{thurston:notes}. See also  \cite[Section 2]{fg:angled-survey}, where holonomy is defined only for normal curves, without any claim of linearity.) The turning angle of 
equation \eqref{eq:turning-angle} is precisely the imaginary part of the (log) holonomy from \cite{thurston:notes}. The issue of finding the right multiple of $2\pi$ is identical to the issue of finding the correct branch of a complex--valued logarithm.

\begin{proof}[Proof of Proposition \ref{prop:holonomy-rep}]
Let $\widetilde{V} \cong \RR^2$ be the universal cover of the torus $V$. The orientation of $V$, the triangulation of $V$, and the generalized angle structure $(\theta)$ all lift to $\widetilde{V}$. We begin by assigning a coherent direction to every oriented edge of $\widetilde{V}$.

\begin{claim}\label{claim:holonomy-1}
Let $X$ be the collection of all oriented edges of $\widetilde{V}$. Then there exists a map $$\psi_\theta:X\rightarrow \mathbb{R}/2\pi\mathbb{Z},$$
assigning each $x \in X$  a direction $\psi_\theta(x)$, 
with the following property. Whenever $x,y\in X$ have the same tail vertex $p$, with $y$ immediately following $x$ for the counterclockwise cyclic order at $p$, and $x'$ denotes the same edge as $x$ with opposite orientation, then 
\begin{equation}\label{eq:psi}
\psi_\theta(x')=\psi_\theta(x)+\pi \quad \text{ and } \quad \psi_\theta(y)=\psi_\theta(x)+\theta_j \, ,
\end{equation}
where $\theta_j$ is the angle between (the projections to $V$ of) $x$ and $y$.
\end{claim}

Consider the graph $\Gamma$ with vertex set $X$ and edges of the form $xx'$ and $xy$ (with $x,x',y$ as in the claim). 
To define the function $\psi_\theta$, it suffices to set $\psi_\theta(x_0) = 0$ for one arbitrary $x_0 \in X$. Then, equation \eqref{eq:psi} gives a way to extend the definition to every other $x \in X$. Because $\Gamma$ is connected, each $x\in X$ will receive (at least one) direction vector. The main content of Claim \ref{claim:holonomy-1} is that the definition is consistent, i.e.\ that no contradictions arise in extending the definition over $\Gamma$.

Up to isomorphism, $\Gamma$ is the infinite trivalent planar graph obtained from the $1$-skeleton of $\widetilde{V}$ by replacing every $v$-valent vertex with a $v$-gon. The complement of $\Gamma$ in the plane has two types of components: 
\begin{enumerate}
\item $v$-gons coming from vertices $p$ of $\widetilde{V}$. The definition is consistent along the boundary of such a $v$-gon because the $v$ angles around $p$ add up to $2\pi\equiv 0~[2\pi]$ for the angle structure $(\theta)$.
\item Hexagons coming from triangles of $\widetilde{V}$ with angles $\theta_i, \theta_j, \theta_k$. The definition is consistent along the boundary of such a hexagon because $\pi+\theta_i+\pi+\theta_j+\pi+\theta_k=4\pi\equiv 0~[2\pi]$ for the angle structure $(\theta)$.
\end{enumerate}
Since $\mathbb{R}^2$ is simply connected, every closed loop in $\Gamma$ bounds some combination of $v$-gons and hexagons. Thus $\psi_\theta$ is well-defined, proving Claim \ref{claim:holonomy-1}. Moreover, given the generalized angle structure $\theta$, the map $\psi_\theta$ is clearly unique up to an additive constant.

\begin{claim}\label{claim:holonomy-2}
Let $T'\widetilde{V}$ denote the space of 
oriented
tangent directions at points of $\widetilde{V}$,
i.e.\ tangent vectors of $T\widetilde{V}$ up to multiplication by positive scalars.
There exists a continuous map $$\Psi_\theta:T'\widetilde{V}\rightarrow \mathbb{R}/2\pi\mathbb{Z}$$ 
that extends $\psi_\theta$, in the sense that for any open edge $e$ of $\widetilde{V}$, if a nonzero tangent vector $[u]\in T'e$ defines on $e$ the orientation of $x\in X$, then $\Psi_\theta([u])=\psi_\theta(x)$.
\end{claim}

We first define $\Psi_\theta$ on one triangle $\tau$
of $\widetilde{V}$: up to a diffeomorphism, identify $\tau$ with 
some fixed equilateral triangle of the plane. For a vector $[u]\in T'\tau$ (pointing into $\tau$ if its root is on $\partial \tau$), let $\alpha([u])\in [0,2\pi)$ be its argument: we may assume that $[u]$ is parallel to one of the three sides of $\tau$ if and only if $\alpha([u])\in \frac{\pi}{3}\mathbb{Z}$. For such $[u]$, declare $\Psi_\theta([u])=\psi_\theta(x)$ where $x$ is the (oriented) edge parallel to $[u]$. For other $[u]$, use interpolation, setting $\Psi_\theta([u])=f_\theta\circ\alpha([u])$ for some continuous function $f_\theta$ such that the net variation of $f_\theta$ between two consecutive multiples of $\frac{\pi}{3}$ is $\theta_i$ or $\theta_j$ or $\theta_k$ (the $(\theta)$-angles of $\tau$). Note that the space of such interpolating maps $f_\theta$ is contractible. This definition is continuous over $T'\tau$ (\emph{i.e.}\ over all arguments) because $\theta_i+\theta_j+\theta_k + \theta_i+\theta_j+\theta_k=2\pi \equiv 0 ~[2\pi]$. 

Similar definitions for other triangles $\tau$ of $\widetilde{V}$ fit together consistently. Moreover, as the direction $[u_t]$ rotates by a full turn (counterclockwise) around a chosen basepoint $p$ of $\widetilde{V}$, the net variation of $\Psi_\theta([u_t])$ is precisely $2\pi$. This follows from the definition for $p$ in the interior of a triangle $\tau$; from $\theta_i+\theta_j+\theta_k=\pi$ when $p$ is in the interior of an edge; and from the fact that the $(\theta)$-angles add up to $2\pi$ around any vertex of $V$ when $p$ is a vertex of $\widetilde{V}$. Claim \ref{claim:holonomy-2} is proved.

\smallskip

The spaces $T'V$ 
(oriented
tangent directions to $V$) and $T'\widetilde{V}$ have the same universal cover $\widetilde{T'V}$, and $\mathbb{R}/2\pi\mathbb{Z}$ has universal cover $\mathbb{R}$. Therefore $\Psi_\theta$ lifts to $\widetilde{\Psi}_\theta:\widetilde{T'V}\rightarrow \mathbb{R}$, and by computing net variation of $\widetilde{\Psi}_\theta$ along a path, this induces a map 
$$\overline{h}_\theta:\pi_1(T'V)\rightarrow \mathbb{R} \,,$$
which is clearly a representation. Its values are linear combinations of the $\theta_i$ (and of $2\pi$) with integer coefficients; therefore it does not depend on the choices made so far (\emph{e.g.}\ of function $f_\theta$). 
Moreover, any non-vanishing vector field on $V$ (for instance, a constant vector field in a flat metric) defines a direction field, which provides every loop $\gamma$ in $V$ with a canonical lift $\overline{\gamma}$ to $T'V$. The value $\overline{h}_\theta(\overline{\gamma})$ does not depend on the choice of direction field, because the angular difference between two such fields is just a scalar function on $V$. Therefore $\overline{h}_\theta$ descends to a representation 
$$\hat{h}_\theta:\pi_1(V) = H_1(V, \ZZ) \rightarrow \mathbb{R} \,.$$

It remains to relate the representation $\hat{h}_\theta$ defined via $T'V$ to the angular holonomy $h_\theta$ defined via turning angles. The connection is as follows. Let $\gamma$ be a smooth, embedded, normal curve in $V$. Then the torus $V$ can be foliated by ``pushoff'' curves isotopic to $\gamma$. The tangent directions to these pushoff curves define a nonzero direction field on $V$. By Claims \ref{claim:holonomy-1} and \ref{claim:holonomy-2}, the tangent direction to $\gamma$ changes by exactly $\mathcal{A}_\gamma-\mathcal{B}_\gamma$ as we walk around $\gamma$. Therefore,
$$ \hat{h}_\theta([\gamma]) \: = \:  \overline{h}_\theta(\overline{\gamma}) \:  = \: \mathcal{A}_\gamma-\mathcal{B}_\gamma \: = \: t_\theta(\gamma).$$
But we have already shown that $\hat{h}_\theta$ is a representation, i.e.\ depends only on the (primitive) homology class $[\gamma]$. Therefore, the turning angle $t_\theta(\gamma)$ of the curve $\gamma$ depends only on the homology class, which implies the holonomy $h_\theta([\gamma])$ is well-defined. This definition of holonomy for primitive classes extends linearly to all of $H_1(V, \ZZ)$, via equation \eqref{eq:holonomy}.

Finally, the linearity of $h_\theta$ follows immediately because $h_\theta = \hat{h}_\theta$ is a representation and $\RR$ is commutative.
\end{proof}

\subsection{How to deform the holonomy}
 
\begin{lemma}\label{lemma:taut-trivial}
The angular holonomy of a veering taut angle structure is always $h_{\text{taut}} = 0$.
\end{lemma}

\begin{proof}
This can be seen from Figure \ref{fig:cuspview}. Indeed, a curve $\gamma$ parallel to the ladderpoles cuts only thin angles off the triangles it crosses, so $\mathcal{A}_\gamma=\mathcal{B}_\gamma=0$ and $h_{\text{taut}}(\gamma)=0$. For more general curves, the main observation is that the triangulation of $\partial\overline{M}$ can be turned into an \emph{oriented} train track, \emph{e.g.}\ in Figure \ref{fig:cuspview} by orienting all rungs from left to right, all blue ladderpoles downward, and all red ladderpoles upward. This way, at each vertex, the two $\pi$-angles separate the incoming edges (branches) from the outgoing ones. Next, consider an oriented curve $\overline{\gamma}$ carried by the train track, and assume for simplicity that $\overline{\gamma}$ consists only of rungs (and visits all ladders in cyclic order from left to right, possibly several times). Since at every vertex of $\overline{\gamma}$ the angles on either side of $\overline{\gamma}$ sum to $\pi$, it is easy to see that $\overline{\gamma}$ can be perturbed to a curve $\gamma$ transverse to the train track, with trivial holonomy $h_{\text{taut}}(\gamma)=0$. Since $h_{\text{taut}}$ is linear on $H_1(\partial \overline{M},\mathbb{R})$, we conclude that $h_{\text{taut}}=0$.
\end{proof}

For non-veering (even taut) angle structures, one may in general have non-trivial holonomy. The key to creating nonzero holonomy is the following relationship between holonomy and deformations.

\begin{lemma}[Lemma 4.4 of \cite{fg:angled-survey}]\label{lemma:holonomy-intersection}
Let $(\theta)$ be a generalized angle structure, and let $\gamma$ and $\delta$ be oriented closed curves on $\bdy M$. Then, for all $t \in \RR$, the deformation $t D^\gamma$ has the following effect on the holonomy of $\delta$:
$$h_{\theta+tD^\gamma}(\delta)=h_\theta(\delta)+2t\cdot \iota(\gamma, \delta)$$
where $\iota$ denotes algebraic intersection number. (For our purposes, the sign convention in $\iota(\gamma, \delta)$ will be irrelevant.)
\end{lemma}


One immediate consequence of Lemma \ref{lemma:holonomy-intersection} is that if $\gamma$ is homologically trivial, or $\gamma$ belongs to a different boundary torus than $\delta$, then $D^\gamma$ does not affect the holonomy of $\delta$ at all. Thus, because $h_{\text{taut}} = 0$ and all the curves $\gamma$ used in Sections \ref{sec:rescuing} and \ref{sec:lowerbound} are trivial, it follows that all the angle structures constructed so far have $h_{\theta} = 0$. To obtain non-trivial holonomy, one must deform along homologically non-trivial curves.

\subsection{Holonomy of the rung direction} \label{sec:holonomy-rungs}

For example, consider a cusp with $2k$ ladders ($k\geq 1$). Let $\delta$ be a closed curve that has intersection number $1$ with the slope of the ladders $\lad_i$, and let $\gamma_1, \dots, \gamma_k$ be consistently oriented curves which travel up \emph{every other} ladder (\emph{i.e.}\ non-adjacent ladders), so that whenever some $\gamma_s$ traverses a triangle, it enters through one of the $0-\pi$ edges and exits through the $0-0$ edge. Then, the deformation $$D:=\sum_{s=1}^k \lambda_s D^{\gamma_s}$$ does not exit the space $\overline{\ang(\tau)}$ of nonnegative angle structures on $\tau$, for small nonnegative $\lambda_s$ (it decreases the $\pi$'s and increases the $0$'s). In fact we can take $\lambda_s=\pi/4$ for all $s$, because each tetrahedron will suffer at most $4$ deformations (the $\gamma_s$ can cross each of its $4$ cusp triangles at most once). Moreover, if we take the $\gamma_s$ in the \emph{other} set of $k$ ladders and reverse their orientation, the same construction works. Moreover still, we can choose to do this on all cusps simultaneously. By Lemma \ref{lemma:holonomy-intersection}, the deformation $D$ perturbs the holonomy of $\delta$ by 
$$2 \cdot \iota \left ( \delta~,~ \textstyle{\sum_{s=1}^k \lambda_s\gamma_s } \right )  =\pm \frac{k\pi}{2}~,$$ 
the sign depending on the choice of ladder set in the cusp containing $\delta$.
We can summarize the construction in the following result:

\begin{proposition} \label{prop:holo-1}
Let $T_1, \dots, T_c$ be the cusps of $M$, carrying $2k_1, \dots, 2k_c$ ladders (annuli) respectively. Choose $\varepsilon_1, \dots, \varepsilon_c \in \{-1,1\}$. Let $\delta_1, \dots, \delta_c$ be homology classes in $H_1(\partial \overline{M}, \mathbb{Z})\simeq \mathbb{Z}^{2c}$ that are $\mathbb{Z}^2$-complements of the slopes of the annuli in $T_1, \dots, T_c$ respectively. Then, there exists a nonnegative angle structure $(\theta)$ such that
\begin{equation}\label{eq:rung-holonomy}
h_\theta(\delta_i)=\frac{\varepsilon_i k_i\pi}{2}
\end{equation}
for every $i\in \{1,\dots, c\}$.
\end{proposition}

Recall that the space $\ang(\tau)$ of positive angle structures on $\tau$ is a convex polytope, whose closure $\overline{\ang(\tau)}$ contains every non-negative angle structure. Thus there are (strictly) positive angle structures with holonomy arbitrarily close to $\varepsilon_i k_i\pi / 2$, for $\varepsilon_i \in \{-1,1\}$. In fact, by convexity of $\overline{\ang(\tau)}$, one can actually take $\varepsilon_i\in [-1,1]$ in equation \eqref{eq:rung-holonomy}, getting a full Cartesian product $\Pi$ of possible holonomies $\left ( h_\theta(\delta_1), \dots, h_\theta(\delta_c)\right )$.

\medskip

A little more can be said if the triangulation of $M$ is transverse--taut. Recall from Observation \ref{obs:transverse} that ascending and descending ladders then alternate in each cusp of $M$, and that each truncated tetrahedron then has precisely two boundary triangles in ascending ladders, and two boundary triangles in descending ladders.

Therefore, if we take all curves $\gamma_i$ along ladders of the \emph{same} type (a notion consistent across all cusps), we can make sure each tetrahedron suffers at most two (not four) deformations, and therefore choose $\lambda_s=\frac{\pi}{2}$ instead of $\frac{\pi}{4}$ in the definition of the deformation $D$ above. However, we can no longer choose the ladder set independently inside each cusp. This is summarized in the following Proposition.

\begin{proposition} \label{prop:holo-2} Suppose $M$ has a veering, \emph{transverse}-taut triangulation with cusp tori $T_1, \dots, T_c$ carrying $2k_1, \dots, 2k_c$ ladders respectively. Take $\varepsilon \in \{-1,1\}$, and let $\delta_1, \dots, \delta_c$ be homology classes in $H_1(\partial \overline{M}, \mathbb{Z})\simeq \mathbb{Z}^{2c}$ that have intersection number $\varepsilon$ with the upwards--oriented ladderpoles in $T_1,\dots, T_c$ respectively.
Then, for any subset $J \subset \{1,\dots, c\}$, there exists a nonnegative angle structure $(\theta)$ such that
$$h_\theta(\delta_i)=\left \{ \begin{array}{ll} \varepsilon k_i\pi & \text{ if $i\in J$;}\\ 0 & \text{ otherwise.}\end{array} \right .$$
\end{proposition}

In summary, Proposition \ref{prop:holo-1} states that the $c$-tuple $\left (h_\theta(\delta_1), \dots, h_\theta(\delta_c)\right )$ can take any value inside some  parallelepiped $\Pi$ centered around the origin of $\mathbb{R}^c$, and Proposition \ref{prop:holo-2} states that in the transverse--taut case, one pair of opposite octants of $\Pi$ can be further homothetized by a factor of 2. Therefore, in the transverse--taut case,  Proposition \ref{prop:holo-2} implies Proposition \ref{prop:holo-1} by averaging out. See Figure \ref{fig:octants}.

\begin{figure}[h!]
\begin{center}
\psfrag{p}{$\Pi$}
\psfrag{h1}{$h_\theta(\delta_1)$}
\psfrag{h2}{$h_\theta(\delta_2)$}
\includegraphics[width=8cm]{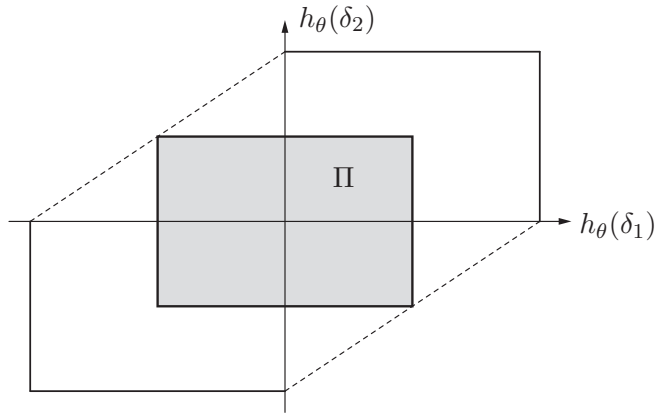}
\caption{Illustration of Propositions \ref{prop:holo-1} and \ref{prop:holo-2}.}
\label{fig:octants}
\end{center}
\end{figure}


\begin{example}

For once-punctured torus bundles, the bounds of Proposition \ref{prop:holo-2} can be seen to be optimal. However, we have found layered triangulations with two cusps, where all points provided by Proposition \ref{prop:holo-2} (for $J\neq 0$) belong to the boundary of the space of achievable holonomies, but where each individual $h_\theta(\delta_1)$ or $h_\theta(\delta_2)$ can take even larger values. 

Consider the punctured rectangle $$[0,2]\times[0,1]\smallsetminus \{0,1,2\}\times \{0,1\}$$ and identify opposite edges to get a twice-punctured torus $T$. The linear map 
$$\varphi_0:=
\begin{bmatrix}  1&1 \\  1&2 \end{bmatrix}^3
=\begin{bmatrix} 5&8 \\  8&13 \end{bmatrix} $$
 preserves $2\mathbb{Z}\oplus \mathbb{Z}$ and therefore induces a mapping class of $T$, still denoted $\varphi_0$ and preserving each puncture. The mapping torus $T\times [0,1]/\sim_{\varphi_0}$ is a 6-fold cover of the once-punctured torus bundle with monodromy $\left[ \begin{smallmatrix}  1&1 \\  1&2  \end{smallmatrix} \right]$ (the figure-8 knot complement) and as such receives a natural layered triangulation with 12 tetrahedra. We alter this construction by defining $$\varphi=\varphi_0 \circ \lambda~,$$ where $\lambda$ is a Dehn twist along the curve $\ell=\left \{\frac{3}{2} \right \}\times [0,1]$ of $T$. Since $\ell$ traverses just one tetrahedron, it turns out we can account for $\lambda$ by just one extra diagonal exchange, so $M:=T\times [0,1]/\varphi$ receives a natural layered (in fact, veering) triangulation with 13 tetrahedra. The number of pairs of annuli is $k_i=2$ for both cusps. Numerical computation shows that if $\delta_1, \delta_2$ are consistently oriented loops around the two punctures, then the space of pairs $(h_\theta(\delta_1), h_\theta(\delta_2))$, as $(\theta)$ runs over all nonnegative angle structures, is the convex hull $Q$ of the six points
$$\pm \left ( \frac{27\pi}{13}, \frac{\pi}{13}\right ), \quad \pm \left ( 2\pi, 2\pi \right ), \quad \pm \left ( \frac{\pi}{13}, \frac{27\pi}{13}\right ) .$$
In particular, the points $\pm (2\pi,2\pi)$, $\pm(0,2\pi)$, and $\pm(2\pi, 0)$ provided by Proposition \ref{prop:holo-2} all belong to the boundary of $Q$, but of course $\frac{27\pi}{13}>2\pi$.
\end{example}

\subsection{Exotic taut angle structures}

If one takes $J=\{1,\dots, c\}$ in Proposition \ref{prop:holo-2}, then every tetrahedron $\Delta$ of $M$ suffers \emph{precisely} two deformations, and these deformations are in the same direction. This is true because whatever the colors of the diagonals in Figure \ref{fig:onetetrahedron}, the triangles cut off at $A$ and $C$ (resp. $B$ and $D$) receive precisely the same colors on all their edges and vertices. As a result, after deformation with $\lambda_s\equiv \frac{\pi}{2}$, each tetrahedron $\Delta$ is flat again. Since we can choose $\varepsilon= 1$ or $\varepsilon= -1$  in Proposition \ref{prop:holo-2}, it follows that 

\begin{proposition}\label{prop:exotic}
Every veering, transverse--taut triangulation comes with at least two ``exotic'' taut angle structures, distinct from the given one (and usually not veering).
\end{proposition}

We have also checked that some, but not all, punctured-torus bundles can admit even more taut angle structures: this happens precisely when the monodromy, as a cyclic word in two letters $R= \left[ \begin{smallmatrix} 1 & 1 \\ 0 & 1 \end{smallmatrix} \right]$ and $L=\left[ \begin{smallmatrix} 1 & 0 \\ 1 & 1  \end{smallmatrix} \right]$, can be decomposed into a product of terms of the form $(RL^*R)(LR^*L)$, where the stars denote arbitrary nonnegative exponents.

\medskip

\subsection{Holonomy of the ladderpole direction} \label{sec:holonomy-poles}

We will find nonnegative angle structures $(\theta)$ that realize large values of $h_\theta([\ell])$, where $[\ell]\in H_1(\partial \overline{M}, \mathbb{R})$ is represented by a curve along the ladderpole direction. By the deformation formula of Lemma \ref{lemma:holonomy-intersection}, this will involve applying deformations $D^\gamma$ for  a curve $\gamma$ that intersects the ladderpoles essentially.

\begin{proposition} \label{prop:ladder-holonomy-1}
Let $M$ be a manifold with $c$ cusps, endowed with a veering triangulation $\tau$. Let $\ell_1, \dots, \ell_c$ be simple closed curves in $\partial \overline{M}$ along the ladderpole directions (with any orientations). Then there exists a nonnegative angle structure $(\theta)$ on $M$ such that $h_{\theta}([\ell_i])=\frac{\pi}{4}$ for all $1\leq i \leq c$.
\end{proposition}

\begin{proposition} \label{prop:ladder-holonomy-2}
Suppose in addition that the triangulation of $M$ is transverse--taut and that the orientations of the curves $\ell_i$ all agree (or all disagree) with the transverse--taut structure. Then for any subset $J$ of $\{1,\dots,c\}$ there exists a nonnegative angle structure $(\theta)$ such that $h_\theta([\ell_i])=\frac{\pi}{2}$ if $i\in J$ and $0$ otherwise.
\end{proposition}

Note that just as with Propositions \ref{prop:holo-1} and \ref{prop:holo-2}, there exists a \emph{positive} angle structure $(\theta) \in \ang(\tau)$ with holonomy arbitrarily close to the values specified above.

Again, Figure \ref{fig:octants} illustrates the relationship between Propositions \ref{prop:ladder-holonomy-1} and \ref{prop:ladder-holonomy-2}. For punctured torus bundles, it is easy to see that Proposition \ref{prop:ladder-holonomy-2} gives the optimal bound when the monodromy has the form $R^aL^b$ for positive integers $a,b$. However, for $R^{a_1}L^{b_1}\dots R^{a_s}L^{b_s}$ the optimal bound becomes $s\frac{\pi}{2}$.

\begin{proof}[Proof of Proposition \ref{prop:ladder-holonomy-1}]
We start by applying Lemma \ref{lemma:hinge-rescue} with the extremal value $t=\frac{\pi}{4}$. The angles of the resulting nonnegative angle structure $( \theta )$ are now as follows:
\begin{itemize}
\item Hinge triangles have angle $0$ at the wide angle, and $\frac{\pi}{2}$ elsewhere.
\item Nonhinge triangles have angle $\pi$ at the thin angle adjacent to the ladderpole segment, and $0$ elsewhere. 
\end{itemize}
On any cusp being considered, the chosen orientation of $\ell$, referred to as ``upward'', induces a partition of the $2k$ ladders into two classes: namely, call \emph{ascending} those ladders whose triangles have their wide angle \emph{above} the base for the orientation of $\ell$, and \emph{descending} the other ladders. (In the absence of a transverse--taut structure, this does not have to be consistent over all $c$ cusps: a tetrahedron may have all its boundary triangles in descending ladders.)

We want to apply the deformation $tD^\gamma$ to the angle structure $( \theta )$, for a carefully chosen curve $\gamma$. Focus on one cusp. Start $\gamma$ by crossing from an ascending ladder $\mathcal{L}$ to a descending one $\mathcal{L}'$, at a hinge of $\mathcal{L}$. Then, let $\gamma$ travel \emph{down} $\mathcal{L}'$ until it enters a hinge, then cross over to the next ascending ladder $\mathcal{L}''$. Let $\gamma$ travel \emph{up} $\mathcal{L}''$ until the first hinge, and so on. See Figure \ref{fig:horizontal-deformation}. Notice that ladders are travelled in the opposite direction compared to Lemma \ref{lemma:nonhinge-rescue}.

\begin{figure}[h!]
\begin{center}
\psfrag{g}{$\gamma$}
\psfrag{U}{$\begin{array}{c} \text{Ascending} \\ \text{ladder $\mathcal{L}''$} \end{array}$}
\psfrag{D}{$\begin{array}{c} \text{Descending} \\ \text{ladder $\mathcal{L}'$} \end{array}$}
\psfrag{DD}{$\begin{array}{c} \text{Descending} \\ \text{ladder} \end{array}$}
\includegraphics[width=12cm]{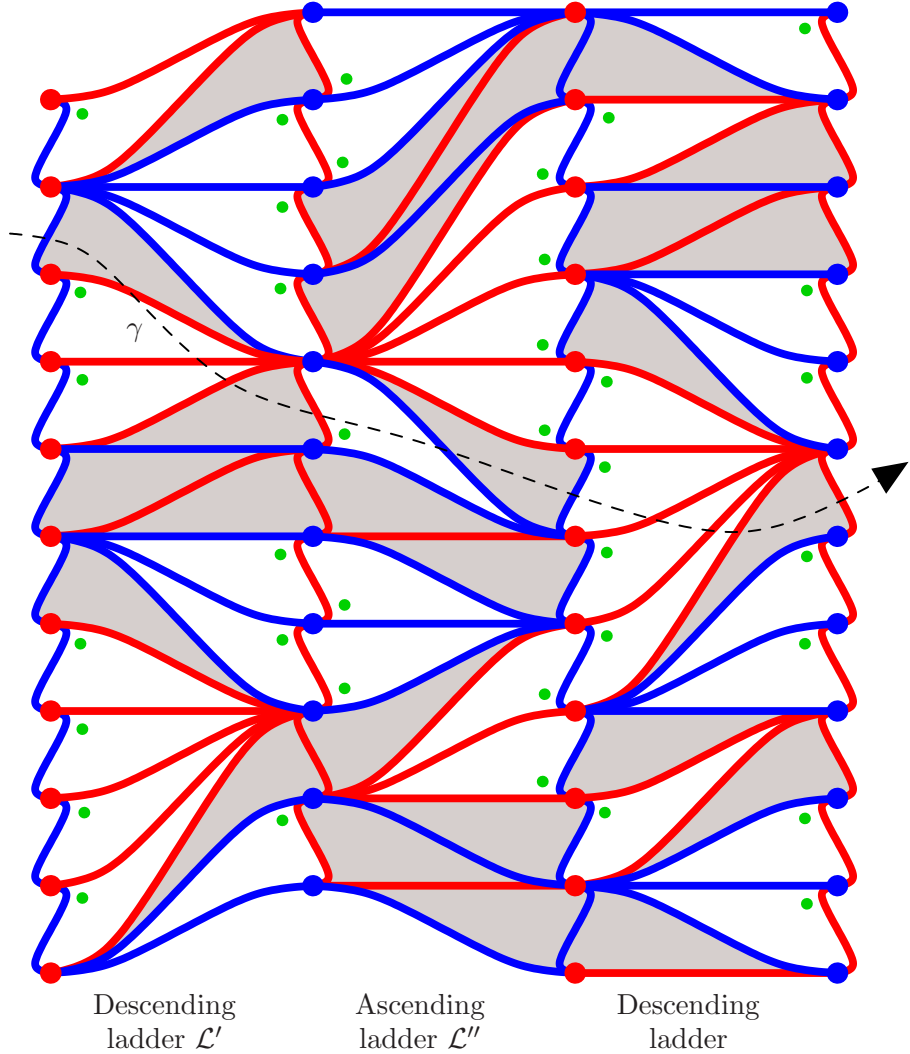}
\caption{The same triangulation as in Figure \ref{fig:cuspview}, with a horizontal deformation curve $\gamma$. Green dots indicate angles of $\pi$ in nonhinge triangles for $( \theta )$; hinge triangles have angles $0,\frac{\pi}{2}, \frac{\pi}{2}$ with the $0$ in the wide corner.}
\label{fig:horizontal-deformation}
\end{center}
\end{figure}

If we orient all ladderpoles consistently with $\ell$, then this curve $\gamma$ always intersects ladderpoles from the same side, and will eventually close up (possibly not at its starting point, in which case we just drop the initial dead arc of $\gamma$). It is easy to check that the  deformation $D^\gamma$ only decreases positive angles for $( \theta )$; moreover we can build one such curve $\gamma_i$ per cusp, crossing the ladderpoles in either direction. Since $\gamma$ always crosses the ladderpoles of each cusp in the same direction, 
$ \iota(\gamma, \ell_i) = m_i \geq 1.$

We claim that 
$$(\theta') = ( \theta )+\sum_{i=1}^c\frac{\pi}{8}D^{\gamma_i}$$ is a nonnegative angle structure.
To see this, just notice that each of the $4n$ triangles of $\partial \overline{M}$ is crossed at most once by the union of the $\gamma_i$. As a result, any angle that gets decreased (and was therefore at least $\frac{\pi}{2}$ for $( \theta )$) is decreased by at most $4\cdot \frac{\pi}{8}$, and thus stays nonnegative. 
By Lemma \ref{lemma:holonomy-intersection}, $h_{\theta'}([\ell_i])= \frac{\pi}{4}m_i$, completing the proof.
\end{proof}

\begin{proof}[Proof of Proposition \ref{prop:ladder-holonomy-2}]
If in addition $M$ is transverse--taut and the $\ell_i$ are consistently oriented, we claim that 
$$( \theta )+\sum_{i\in J} \frac{\pi}{4}D^{\gamma_i}$$ 
is still nonnegative. Of course, $\bigcup_{i\in J}\gamma_i$ still crosses each triangle of $\partial \overline{M}$ at most once.

The claim is easy to check for a nonhinge tetrahedron $\Delta$: the $\pi$-angle of $\Delta$ gets decreased at most four times by $\frac{\pi}{4}$ (once per cusp triangle of $\Delta$) while the other angles only get increased.

For $\Delta$ a hinge tetrahedron, we must discuss two possible cases. Note that (without loss of generality, up to exchanging colors), the curves $\gamma_i$ always cross from one ladder to the next at a hinge tetrahedron $\Delta$ with the upper diagonal red and the lower diagonal blue (this is true in Figure \ref{fig:horizontal-deformation}; recall the diagonals of a tetrahedron are seen as wide vertices of the corresponding cusp triangles). Call such tetrahedra \emph{hinges of type 1}. When $\gamma_i$ exits a \emph{descending} ladder, it decreases the $\frac{\pi}{2}$-angle \emph{clockwise} from the wide angle in a hinge of type 1. When $\gamma_i$ exits an \emph{ascending} ladder, it decreases the $\frac{\pi}{2}$-angle \emph{counterclockwise} from the wide angle in a hinge of type 1. Moreover, $\bigcup_{i\in J}\gamma_i$ visits hinges of type 1 \emph{only} as the $\gamma_i$ leave a ladder for an adjacent one. Therefore, since $\Delta$ has only two tips in ascending (resp. descending) ladders, each $\frac{\pi}{2}$-angle of $\Delta$ is decreased at most twice by $\frac{\pi}{4}$ and thus stays nonnegative.

For a hinge tetrahedron $\Delta'$ of the other type (upper diagonal blue and lower diagonal red), the argument is similar: cusp triangles of $\Delta'$ are only (possibly) visited by $\bigcup_{i\in J}\gamma_i$ as $\gamma_i$ \emph{enters} a new ladder (the rest of the time $\gamma_i$ travels up or down fans); for example in Figure \ref{fig:horizontal-deformation} the very first visible triangle crossed by $\gamma$ is such a hinge. Again, $\gamma_i$ decreases the $\frac{\pi}{2}$-angle of $\Delta'$ clockwise from the wide angle if it encounters $\Delta'$ in a descending ladder, and decreases the other $\frac{\pi}{2}$-angle of $\Delta'$ if it encounters $\Delta'$ in an ascending ladder. Since $\Delta'$ has exactly two cusp triangles in each type of ladder, $\Delta'$ is still nonnegative after applying $\sum_{i\in J}\frac{\pi}{4}D^{\gamma_i}$.
\end{proof}

\begin{remark}
A little more work would show that the slope of the curve $\gamma$ (like that of the ladderpoles) is also an invariant of the veering structure, \emph{i.e.}\ does not depend on the starting point and starting direction chosen to construct $\gamma$. In the layered (veering) case, it would be interesting to relate this slope to the one defined by the fiber.
\end{remark}

\bibliographystyle{hamsplain}
\bibliography{biblio}

\providecommand{\bysame}{\leavevmode\hbox to3em{\hrulefill}\thinspace}
\providecommand{\href}[2]{#2}
\begin{thebibliography}{10}

\bibitem{agol}
Ian Agol, \emph{Ideal triangulations of pseudo-{A}nosov mapping tori}, Topology
  and geometry in dimension three, Contemp. Math., vol. 560, Amer. Math. Soc.,
  Providence, RI, 2011, pp.~1--17.

\bibitem{cooper-tillmann}
Daryl Cooper and Stephan Tillmann, \emph{The {T}hurston norm via normal
  surfaces}, Pacific J. Math. \textbf{239} (2009), no.~1, 1--15.

\bibitem{fathi}
Albert Fathi, \emph{D\'emonstration d'un th\'eor\`eme de {P}enner sur la
  composition des twists de {D}ehn}, Bull. Soc. Math. France \textbf{120}
  (1992), no.~4, 467--484.

\bibitem{fg:arborescent}
David Futer and Fran\c{c}ois Gu\'eritaud, \emph{Angled decompositions of
  arborescent link complements}, Proc. London Math. Soc. \textbf{98} (2009),
  no.~2, 325--364.

\bibitem{fg:angled-survey}
\bysame, \emph{From angled triangulations to hyperbolic structures},
  Interactions between hyperbolic geometry, quantum topology and number theory,
  Contemp. Math., vol. 541, Amer. Math. Soc., Providence, RI, 2011,
  pp.~159--182.

\bibitem{gf:punctured-torus}
Fran\c{c}ois Gu\'eritaud and David Futer~(appendix), \emph{On canonical
  triangulations of once-punctured torus bundles and two-bridge link
  complements}, Geom. Topol. \textbf{10} (2006), 1239--1284.

\bibitem{ozzie}
Craig~D. Hodgson, J.~Hyam Rubinstein, Henry Segerman, and Stephan Tillmann,
  \emph{Veering triangulations admit strict angle structures}, Geom. Topol.
  \textbf{15} (2011), 2073--2089.

\bibitem{kang-rubinstein:taut-angle}
Ensil Kang and J.~Hyam Rubinstein, \emph{Ideal triangulations of 3-manifolds.
  {II}. {T}aut and angle structures}, Algebr. Geom. Topol. \textbf{5} (2005),
  1505--1533.

\bibitem{lackenby:taut}
Marc Lackenby, \emph{Taut ideal triangulations of 3-manifolds}, Geom. Topol.
  \textbf{4} (2000), 369--395.

\bibitem{lackenby:surgery}
\bysame, \emph{Word hyperbolic {D}ehn surgery}, Invent. Math. \textbf{140}
  (2000), no.~2, 243--282.

\bibitem{lackenby:heegaard-alg}
\bysame, \emph{An algorithm to determine the {H}eegaard genus of simple
  3-manifolds with nonempty boundary}, Algebr. Geom. Topol. \textbf{8} (2008),
  no.~2, 911--934.

\bibitem{luo-tillmann}
Feng Luo and Stephan Tillmann, \emph{Angle structures and normal surfaces},
  Trans. Amer. Math. Soc. \textbf{360} (2008), 2849--2866.

\bibitem{rivin:volume}
Igor Rivin, \emph{Euclidean structures on simplicial surfaces and hyperbolic
  volume}, Ann. of Math. (2) \textbf{139} (1994), no.~3, 553--580.

\bibitem{rivin:combinatorial}
\bysame, \emph{Combinatorial optimization in geometry}, Adv. in Appl. Math.
  \textbf{31} (2003), no.~1, 242--271.

\bibitem{thurston:notes}
William~P. Thurston, \emph{The geometry and topology of three-manifolds},
  Princeton Univ. Math. Dept. Notes, 1980, Available at {\tt
  http://www.msri.org/gt3m/}.

\end{thebibliography}

\end{document}